\documentclass[10pt]{amsart}
\usepackage{amsrefs}
\usepackage{amssymb}
\usepackage[all]{xy}

\usepackage{hyperref}

\usepackage{tikz}
\usetikzlibrary{matrix,arrows}

\DeclareMathOperator{\Gal}{Gal}

\DeclareMathOperator{\Img}{Im}
\DeclareMathOperator{\Inf}{inf}

\DeclareMathOperator{\Ker}{Ker}

\DeclareMathOperator{\res}{res}

\DeclareMathOperator{\dd}{d}

\DeclareFontFamily{U}{wncy}{}
\DeclareFontShape{U}{wncy}{m}{n}{<->wncyr10}{}
\DeclareSymbolFont{mcy}{U}{wncy}{m}{n}
\DeclareMathSymbol{\Sha}{\mathord}{mcy}{"58}
\DeclareMathSymbol{\sha}{\mathord}{mcy}{"78}

\begin{document}

\newtheorem{thm}{Theorem}[section]
\newtheorem{cor}[thm]{Corollary}
\newtheorem{lem}[thm]{Lemma}
\newtheorem{fact}[thm]{Fact}
\newtheorem{prop}[thm]{Proposition}
\newtheorem{defin}[thm]{Definition}
\newtheorem{exam}[thm]{Example}
\newtheorem{examples}[thm]{Examples}
\newtheorem{rem}[thm]{Remark}
\newtheorem{case}{\sl Case}
\newtheorem{claim}{Claim}
\newtheorem{question}[thm]{Question}
\newtheorem{conj}[thm]{Conjecture}
\newtheorem*{notation}{Notation}
\swapnumbers
\newtheorem{rems}[thm]{Remarks}
\newtheorem*{acknowledgment}{Acknowledgements}
\newtheorem*{thmno}{Theorem}

\newtheorem{questions}[thm]{Questions}
\numberwithin{equation}{section}

\newcommand{\gr}{\mathrm{gr}}
\newcommand{\inv}{^{-1}}
\newcommand{\isom}{\cong}
\newcommand{\dbC}{\mathbb{C}}
\newcommand{\F}{\mathbb{F}}
\newcommand{\dbN}{\mathbb{N}}
\newcommand{\Q}{\mathbb{Q}}
\newcommand{\dbR}{\mathbb{R}}
\newcommand{\dbU}{\mathbb{U}}
\newcommand{\Z}{\mathbb{Z}}
\newcommand{\calG}{\mathcal{G}}
\newcommand{\K}{\mathbb{K}}
\newcommand{\rmH}{\mathrm{H}}
\newcommand{\bfH}{\mathbf{H}}
\newcommand{\bfLam}{\mathbf{\Lambda}}
\newcommand{\calX}{\mathcal{X}}
\newcommand{\calY}{\mathcal{Y}}
\newcommand{\calV}{\mathcal{V}}
\newcommand{\calE}{\mathcal{E}}
\newcommand{\calW}{\mathcal{W}}
\newcommand{\rmr}{\mathrm{r}}
\newcommand{\Span}{\mathrm{Span}}
\newcommand{\rmQ}{\mathrm{Q}}
\newcommand{\rmT}{\mathrm{T}}
\newcommand{\rmq}{\mathrm{q}}


\newcommand{\hac}{\hat c}
\newcommand{\hatheta}{\hat\theta}

\title[Pro-$p$ groups of $p$-absolute Galois type]{Groups of $p$-absolute Galois type \\ that are not absolute Galois groups}
\author{Simone Blumer}
\address{Department of Mathematics and Applications, University of Milano Bicocca, 20125 Milan, Italy EU --- Department of Mathematics, University of Zaragoza, 50009 Zaragoza, Spain EU}
\email{s.blumer@campus.unimib.it}
\author{Alberto Cassella}
\address{Department of Mathematics and Applications, University of Milano Bicocca, 20125 Milan, Italy EU --- Department of Mathematics, University of Zaragoza, 50009 Zaragoza, Spain EU}
\email{a.cassella@campus.unimib.it}
\author{Claudio Quadrelli}
\address{Department of Science and High-Tech, University of Insubria, 22100 Como, Italy EU}
\email{claudio.quadrelli@uninsubria.it}
\date{\today}
\dedicatory{To Pablo Spiga, an enthusiast algebraist \\ and a ``graphomaniac'', with admiration.}

\begin{abstract}
 Let $p$ be a prime. We study pro-$p$ groups of $p$-absolute Galois type, as defined by Lam--Liu--Sharifi--Wake--Wang.
 We prove that the pro-$p$ completion of the right-angled Artin group associated to a chordal simplicial graph is of $p$-absolute Galois type, and moreover it satisfies a strong version of the Massey vanishing property.
 Also, we prove that Demushkin groups are of $p$-absolute Galois type, and that the free pro-$p$ product --- and, under certain conditions, the direct product --- of two pro-$p$ groups of $p$-absolute Galois type satisfying the Massey vanishing property, is again a pro-$p$ group of $p$-absolute Galois type satisfying the Massey vanishing property.
 Consequently, there is a plethora of pro-$p$ groups of $p$-absolute Galois type satisfying the Massey vanishing property that do not occur as absolute Galois groups.
\end{abstract}



\maketitle

\section{Introduction}
\label{sec:intro}
Throughout the paper, $p$ will denote a prime number.
Given a field $\K$, let $\bar\K_s$ denote the separable closure of $\K$, and let $\K(p)$ denote the maximal $p$-extension of $\K$.
The {\sl absolute Galois group} $G_{\K}:=\Gal(\bar\K_s/\K)$ is a profinite group, and the Galois group $G_{\K}(p):=\Gal(\K(p)/\K)$, called the {\sl maximal pro-$p$ Galois group} of $\K$, is the maximal pro-$p$ quotient of $G_{\K}$.
A major difficult problem in Galois theory is the characterization of profinite groups which occur as absolute Galois groups of fields, and of pro-$p$ groups which occur as maximal pro-$p$ Galois groups (see, e.g., \cite[\S~3.12]{WDG} and \cite[\S~2.2]{birs}).
Observe that if a pro-$p$ group $G$ does not occur as the maximal pro-$p$ Galois group of a field containing a root of 1 of order $p$, then it does not occur as the absolute Galois group of any field (see, e.g., \cite[Rem.~3.3]{cq:noGal}).
For this reason, the pursue of obstructions which detect effectively pro-$p$ groups which do not occur as absolute Galois groups has great prominence in current research in Galois theory (see, e.g., \cite{BLMS,cem,eq:kummer,cq:noGal}).

The celebrated {\sl Bloch-Kato Conjecture} --- established by M.~Rost and V.~Voevodsky, with Ch.~Weibel's ``patch'', and now called the Norm Residue Theorem (see \cite{rost,voev, weibel,HW:book}) --- provides a description of the Galois cohomology of absolute Galois groups in terms of low-degree cohomology.
As a consequence, if $\K$ is a field containing a root of 1 of order $p$, the structure of the {\sl $\Z/p$-cohomology algebra} 
$$\bfH^\bullet(G_{\K}(p)):=\coprod_{n\geq0}\rmH^n(G_{\K}(p),\Z/p),$$
endowed with the {\sl cup-product} $\smallsmile$, is determined by degrees 1 and 2.
This remarkable results provided new fuel --- and new substantial results --- to the research on maximal pro-$p$ Galois groups of fields.

Subsequently, the paper by M.~Hopkins and K.~Wickelgren \cite{hopwick} kicked off a hectic research on {\sl Massey products} in Galois cohomology.
Given a pro-$p$ group $G$ and an integer $n\geq2$, the {\sl $n$-fold Massey product} is a multi-valued map which associates a (possibly empty) subset of $\rmH^2(G,\Z/p)$ to a $n$-tuple of elements of $\rmH^1(G,\Z/p)$ (if $n=2$ it coincides with the cup-product).
Moreover, $G$ is said to satisfy the {\sl $n$-Massey vanishing property} if every non-empty value of an $n$-fold Massey product contains 0. 
{In \cite{eli:massey}, E.~Matzri proved that if $\K$ is a field containing a root of 1 of order $p$, then $G_{\K}(p)$ has the 3-Massey vanishing property (see also \cite{EM:massey} and \cite{MT:masseyall}}): this result produced new obstructions for the realization of pro-$p$ groups as absolute Galois groups (see, e.g., \cite[\S~7]{mt:massey}); and Mina\v{c} and T\^an conjectured that such a $G_{\K}(p)$ has the $n$-Massey vanishing property for {\sl every} integer $n\geq 3$ (see \cite[Conj.~1.1]{mt:conj}).

Another cohomological property enjoyed by maximal pro-$p$ Galois groups --- and related to both the Norm Residue Theorem and Massey products --- is the one which gives the title to this paper.
A pro-$p$ group $G$ is said to be {\sl of $p$-absolute Galois type} if, for every $\alpha\in \rmH^1(G,\Z/p)$, the sequence
\begin{equation}\label{eq:exactsequence intro}
 \xymatrix@C=1.3truecm{\rmH^1(N,\Z/p)\ar[r]^-{\mathrm{cor}_{N,G}^1} & \rmH^1(G,\Z/p)\ar[r]^-{\textvisiblespace\smallsmile\alpha} &
 \rmH^2(G,\Z/p)\ar[r]^-{\res_{G,N}^2} & \rmH^2(N,\Z/p)}
\end{equation}
is exact, where $N=\Ker(\alpha)$, and the middle arrow denotes the cup-product by $\alpha$ (see \cite[\S~1.4]{LLSWW}).
This condition, generalized to arbitrary cohomological degrees, is satisfied by the absolute Galois group of a field containing a root of 1 of order $p$, and it is heavily used in the proof of the Norm Residue Theorem (see \cite[Thm.~3.6]{HW:book}).
Y.H.J~Lam, Y.~Liu, R.T.~Sharifi, P.~Wake, and J.~Wang proved that the sequence \eqref{eq:exactsequence intro} is exact at $\rmH^1(G,\Z/p)$ if, and only if, the $p$-fold Massey product associated to the $p$-tuple $\alpha,\ldots, \alpha,\beta$ (where $\alpha$ appears $p-1$ times) contains 0 whenever $\alpha,\beta$ are elements of $\rmH^1(G,\Z/p)$ such that the cup-product $\alpha\smallsmile\beta$ is trivial; and moreover if $G$ is of $p$-absolute Galois type then it has the 3-Massey vanishing property (see \cite{LLSWW} and \S~\ref{ssec:pabsgaltype}).
 
It is natural to ask how well pro-$p$ groups satisfying these cohomological properties --- the Massey vanishing properties and being of $p$-absolute Galois type --- ``approximate'' maximal pro-$p$ Galois groups of fields containing a root of 1 of order $p$: i.e., if there are (and how ``many'') pro-$p$ groups satisfying these cohomological properties but which do not occur as maximal pro-$p$ Galois groups of fields containing a root of 1 of order $p$.

Keeping this question in mind, we focus on the family of {\sl right-angled Artin pro-$p$ groups} (pro-$p$ RAAGs for short).
The pro-$p$ RAAG $G_\Gamma$ associated to a simplicial graph $\Gamma$ is the pro-$p$ completion of the discrete right-angled Artin group associated to $\Gamma$.
Right-angled Artin groups have surprising richness and flexibility, and played a prominent role in geometric group theory in recent decades (for an overview on right-angled Artin groups, see \cite{RAAGs}).
Moreover, pro-$p$ RAAGs share several properties of their discrete brothers (see, e.g., \cite{lorensen,KW:raags}): in particular, they have a very rich subgroup structure, and their $\Z/p$-cohomology algebra depends only on degrees 1 and 2.
For this reasons, pro-$p$ RAAGs are extremely interesting also from a Galois-theoretic perspective.

Our first goal is to prove that every pro-$p$ RAAG satisfies a strong Massey vanishing property, introduced by
A.~P\'al and E.~Szab\'o in \cite{pal:massey} (see also \cite[\S~4]{JT:U4}).

\begin{thm}\label{thm:massey intro}
 Let $\Gamma$ be a simplicial graph, and let $G_\Gamma$ be the associated pro-$p$ RAAG.
 For every integer $n\geq3$, $G_\Gamma$ has the strong $n$-Massey vanishing property, i.e., for every $n$-tuple $\alpha_1,\ldots,\alpha_n$ of elements of $\rmH^1(G_\Gamma,\Z/p)$ such that the $n-1$ cup-products $$\alpha_1\smallsmile\alpha_2,\;\alpha_2\smallsmile\alpha_3,\;\ldots,\;\alpha_{n-1}\smallsmile\alpha_n$$ are trivial, the associated $n$-fold Massey product contains 0. 
\end{thm}

Observe that for a general pro-$p$ group $G$, the condition on the triviality of the cup-products $\alpha_1\smallsmile\alpha_2,\ldots,\alpha_{n-1}\smallsmile\alpha_n$ is a {\sl necessary} condition for the non-emptiness of the $n$-fold Massey product associated to the $n$-tuple $\alpha_1,\ldots,\alpha_n\in \rmH^1(G,\Z/p)$ (see, e.g., \cite[\S~2]{mt:massey} and Proposition~\ref{prop:massey cup} below).
By Theorem~\ref{thm:massey intro}, this condition is also {\sl sufficient} if $G$ is a pro-$p$ RAAG.

Our second goal is to show that a wide family of simplicial graphs yields pro-$p$ RAAGs of $p$-absolute Galois type.
Recall that a simplicial graph is said to be {\sl chordal} (or {\sl triangulated}) if each of its cycles of length at least 4 has a chord, i.e. if it contains no induced cycles other than triangles.
We prove the following.

\begin{thm}\label{thm:AbsGalType intro}
 Let $\Gamma$ be a simplicial graph, and let $G_\Gamma$ be the associated pro-$p$ RAAG.
 Then $G_\Gamma$ is a pro-$p$ group of $p$-absolute Galois type in the following cases:
 \begin{itemize}
  \item[(i)] if $\Gamma$ is chordal;
  \item[(ii)] if $\Gamma$ consist of a row of subsequent squares, i.e., $\Gamma$ has geometric realization
  
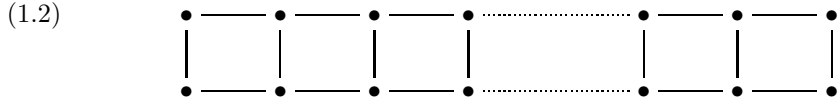
\begin{equation}\label{eq:squaregraph intro}
     \xymatrix@R=1.5pt{ \bullet\ar@{-}[r] \ar@{-}[ddd]& \bullet\ar@{-}[r] \ar@{-}[ddd]& \bullet\ar@{-}[r] \ar@{-}[ddd] & 
    \bullet\ar@{.}[rr] \ar@{-}[ddd] && \bullet\ar@{-}[r] \ar@{-}[ddd] &  \bullet\ar@{-}[r] \ar@{-}[ddd] & \bullet\ar@{-}[ddd]  \\  \\  
  \\ \bullet\ar@{-}[r] &\bullet\ar@{-}[r] &\bullet\ar@{-}[r] & \bullet\ar@{.}[rr]  &&\bullet\ar@{-}[r] &\bullet\ar@{-}[r]& \bullet  }   \end{equation}
 \end{itemize}
\end{thm}

An example of pro-$p$ RAAG associated to a chordal simplicial graph is the pro-$p$ group
\[
G=\left\langle\: v_1,\:v_2,\:\ldots,\: v_d\:\mid\:[v_1,v_2]=[v_2,v_3]=\ldots=[v_{d-1},v_d]=1\:\right\rangle,
\]
which is the pro-$p$ RAAG associated to the simplicial graph $\mathrm{L}_{d-1}$ with geometric realization
\[  \mathrm{L}_{d-1}\qquad\xymatrix{\bullet\ar@{-}[r] &\bullet\ar@{-}[r] & \bullet\ar@{.}[r]  &\bullet\ar@{-}[r] &\bullet}
\]
with $d$ vertices and $d-1$ edges.

{
	A simplicial graph $\Gamma$ is said to be {\sl of elementary type} if no induced subgraph of $\Gamma$ has either of the two forms
	
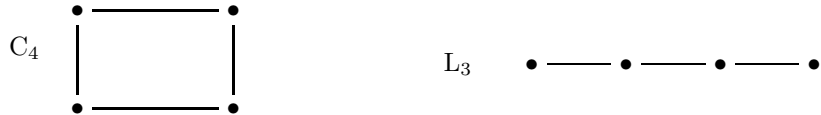
\[
	\xymatrix@R=1.5pt{ \bullet\ar@{-}[r] \ar@{-}[ddd] & \bullet\ar@{-}[ddd] \\
		\mathrm{C}_4\qquad\qquad && \\  \\ \bullet\ar@{-}[r] & \bullet}
	\qquad\qquad
	\xymatrix@R=1.5pt{ \\ \\ \mathrm{L}_3\qquad\bullet\ar@{-}[r] & \bullet\ar@{-}[r]  & \bullet\ar@{-}[r] & \bullet}
	\]
	--- simplicial graphs of elementary type are sometimes called {\sl Droms graphs}, as C.~Droms showed that these are precisely the simplicial graphs such that all subgroups of the associated RAAGs are again RAAGs (see \cite{droms:etc}).
	In \cite{sz:raags}, I.~Snopce and P.A.~Zalesski{\u{\i}} proved that given a simplicial graph $\Gamma$ with associated pro-$p$ RAAG $G_\Gamma$, one has $G_\Gamma\simeq G_{\K}(p)$ for some field $\K$ containing a root of 1 of order $p$ if, and only if,
	$\Gamma$ is of elementary type (see \cite[Thm.~1.2]{sz:raags}).
	For example,} every simplicial graph as in Theorem~\ref{thm:AbsGalType intro}--(ii) is not of elementary type, and thus the associated pro-$p$ RAAG $G$ does not occur as the maximal pro-$p$ Galois group of a field containing a root of 1 of order $p$.

Observe that every simplicial graph of elementary type is chordal, but chordal simplicial graphs containing a length-3 path as an induced subgraph is not of elementary type --- e.g., if $d\geq4$ the simplicial graph $\mathrm{L}_d$ is not of elementary type, and thus the associated pro-$p$ RAAG $G$ does not occur as the maximal pro-$p$ Galois group of a field containing a root of 1 of order $p$.

Then, we turn our attention to other sources of pro-$p$ groups of $p$-absoulte Galois type.
Besides pro-$p$ RAAGs, also {\sl Demushkin groups} are pro-$p$ groups of $p$-absolute Galois type, independently on their realizability as maximal pro-$p$ Galois groups of fields (see Remark~\ref{rem:Demushkin Galois}).

\begin{thm}\label{thm:Demushkin intro}
 Let $G$ be a Demushkin group. Then $G$ is of $p$-absolute Galois type.
\end{thm}

Moreover, one may employ free pro-$p$ products and direct products to combine pro-$p$ groups of $p$-absolute Galois type, and obtain new pro-$p$ groups of $p$-absolute Galois type.

\begin{thm}\label{thm:products intro}
Let $G_1,G_2$ be two pro-$p$ groups of $p$-absolute Galois type.
\begin{itemize}
 \item[(i)] The free pro-$p$ product $G_1\amalg G_2$ of $G_1$ and $G_2$ is a pro-$p$ group of $p$-absolute Galois type.
 \item[(ii)] Assume further that for both $i=1,2$: {\rm (a)} $G_i$ is finitely generated; {\rm (b)} the abelianization of $G_i$ is a free abelian pro-$p$ group; and {\rm (c)} $\rmH^2(G_i,\Z/p)$ is generated by cup-products of elements of $\rmH^1(G_i,\Z/p)$.
 Then also the direct product $G_1\times G_2$ of $G_1$ and $G_2$ is a pro-$p$ group of $p$-absolute Galois type.
\end{itemize} 
\end{thm}

Notice that pro-$p$ RAAGs, and the pro-$p$ completions of orientable surface groups (which are Demushkin groups) satisfy the three conditions (a)--(c) prescribed in  Theorem~\ref{thm:products intro}--(ii).
Analogously, we prove that also the $n$-Massey vanishing property, for every $n\geq 3$, is preserved by direct products, under the same conditions (a)--(c) as in Theorem~\ref{thm:products intro}--(ii), see Theorem~\ref{thm:directprod massey}.
Incidentally, this implies that a positive solution of Efrat's {\sl Elementary Type Conjecture} implies a positive solution to Mina\v{c}-T\^an's Massey vanishing conjecture for fields containing all roots of 1 of $p$-power order whose maximal pro-$p$ Galois group is finitely generated --- see Corollary~\ref{cor:ETC} ---, and this provides a strong evidence for the latter conjecture.

It is worth underlining that the direct product of two pro-$p$ groups may occur as the maximal pro-$p$ Galois group of a field containing a root of 1 of order $p$ only if both factors occur as the maximal pro-$p$ Galois group of fields containing a root of 1 of order $p$, and one of the two factor is a free abelian pro-$p$ group (see \cite[Prop.~3.2]{koenig}). 
Therefore, Theorem~\ref{thm:AbsGalType intro} and Theorem~\ref{thm:products intro} produce a lot of concrete examples of pro-$p$ groups of $p$-absolute Galois type which do not occur as the maximal pro-$p$ Galois group of a field containing a root of 1 of order $p$, and hence neither as an absolute Galois group.

Altogether, one concludes the following.

\begin{cor}\label{cor:noGal}
 There exist a lot of pro-$p$ groups of $p$-absolute Galois type with the $n$-Massey vanishing property, for every $n\geq3$, that do not occur as maximal pro-$p$ Galois groups of fields containing a root of 1 of order $p$, and hence neither as absolute Galois groups.
\end{cor}

From Corollary~\ref{cor:noGal}, one sees that the property of being of $p$-absolute Galois type and the Massey vanishing properties, used to filter out maximal pro-$p$ Galois groups of fields containing a root of 1 of order $p$ (and thus also absolute Galois pro-$p$ groups) from the class of pro-$p$ groups provide --- even combined together --- a strainer whose mesh is rather coarse.
The result of Snopce and Zalesski{\u{\i}} is based on the study of the {\sl Bloch-Kato property} and of {\sl 1-cyclotomicity} (see \S~\ref{ssec:Gal} below): so, these two properties --- which are consequences of the Norm Residue Theorem and of Kummer theory respectively --- appear to be much more restrictive, and effective for the pursue of pro-$p$ groups that are not absolute Galois groups.
In fact, the strength of these two properties lies in the fact that they are hereditary with respect to closed subgroups.
Therefore, it would be interesting to investigate pro-$p$ groups such that {\sl every closed subgroup} is of $p$-absolute Galois type. 
At this aim, we prove that it is enough to verify that every {\sl open} subgroup is of $p$-absolute Galois type (see Proposition~\ref{prop:hereditarily}).

\medskip
{\small The paper is structured as follows.
In \S~\ref{sec:cohom} we list some facts on $\Z/p$-cohomology of pro-$p$ groups (cf. \S~\ref{ssec:cohomology}), and some properties of simplicial graphs and pro-$p$ RAAGs (cf. \S~\ref{ssec:graphs}--\ref{ssec:raags cohom}), which are preliminary to the proofs of our results.
In \S~\ref{sec:massey} we give a brief (and self-contained) tractation on Massey products in $\Z/p$-cohomology of pro-$p$ groups, and we prove Theorem~\ref{thm:massey intro} (cf. \S~\ref{ssec:massey raags}).
In \S~\ref{sec:raags pabsgaltype} we prove Theorem~\ref{thm:AbsGalType intro} (cf. \S~\ref{ssec:thm1}), after some preliminary technical results whose proofs mix together combinatorics and group cohomology (cf. \S~\ref{ssec:cup product}--\ref{ssec:res}).
Finally, in \S~\ref{sec:prod} we deal with free pro-$p$ products and direct products (cf. \S~\ref{ssec:freeprod} and \S~\ref{ssec:directprod} respectively) and with Demushkin groups (cf. \S~\ref{ssec:Demushkin}), and we prove Theorems~\ref{thm:Demushkin intro}--\ref{thm:products intro}; while in \S~\ref{ssec:hered} we define pro-$p$ groups hereditarily of $p$-absolute Galois type, and we prove Proposition~\ref{prop:hereditarily}.
}

\section{Pro-$p$ RAAGs and cohomology}\label{sec:cohom}

We work in the world of pro-$p$ groups.
Henceforth, every subgroup of a pro-$p$ group will be tacitly assumed to be closed, and the generators of a 
subgroup will be intended in the topological sense.
For a pro-$p$ group $G$ and a positive integer $n$, $G^n$ will denote the subgroup of $G$ generated by the $n$-th powers of all elements of $G$.
Moreover, for two elements $g,h\in G$, we set $${}^gh=ghg^{-1},\qquad\text{and}\qquad[g,h]={}^{g}h\cdot  h^{-1},$$
and for two subgroups $H_1,H_2$ of $G$, $[H_1,H_2]$ will denote the subgroup of $G$ generated by all commutators $[g,h]$
with $g\in H_1$ and $h\in H_2$.
In particular, $G'$ will denote the closure of the commutator subgroup of $G$, and $\Phi(G)$ will denote the {\sl Frattini subgroup} of $G$, i.e., $\Phi(G)=G^p\cdot G'$. 


\subsection{Preliminaries on pro-$p$ groups and cohomology}\label{ssec:cohomology}

For the definition and properties of $\Z/p$-cohomology of pro-$p$ groups, we refer to \cite[Ch.~I, \S~4]{serre:galc} and to \cite[Ch.~III, \S~9]{nsw:cohn}.
The definition of the cup-product may be found in \cite[Ch.~I, \S~4]{nsw:cohn}.

Let $G$ be a pro-$p$ group, and consider $\Z/p$ as a trivial $G$-module.
Then 
\begin{equation}\label{eq:H1}
 \rmH^1(G,\Z/p)=\mathrm{Hom}(G,\Z/p)\simeq(G/\Phi(G))^\ast,
\end{equation}
where the second term is the group of homomorphisms of pro-$p$ groups from $G$ to $\Z/p$, and $\textvisiblespace^\ast$ denotes the dual as a $\Z/p$-vector space.
Thus, if $G$ is finitely generated and $\calX=\{x_1,\ldots,x_d\}$ is a minimal generating set of $G$, then $\rmH^1(G,\Z/p)$ has a basis $\calX^\ast=\{\chi_1,\ldots,\chi_d\}$ dual to $\calX$, i.e., $\chi_i(x_j)=\delta_{ij}$ for every $i,j\in\{1,\ldots,d\}$.

A short exact sequence of pro-$p$ groups
\begin{equation}\label{eq:presentation}
 \xymatrix{  \{1\} \ar[r] & R\ar[r] &F \ar[r] & G\ar[r] & \{1\}}
\end{equation}
is said to be a {\sl minimal presentation of $G$} if $F$ is a free pro-$p$ group and $R\subseteq\Phi(F)$ or, equivalently, if the epimorphism $F\twoheadrightarrow G$ induces an isomorphism
$$\Inf_{G,F}^1\colon \rmH^1(G,\Z/p)\overset{\sim}{\longrightarrow} \rmH^1(F,\Z/p).$$
The elements of $R$ are called {\sl relations} of $G$, and a minimal set generating $R$ as a normal subgroup of $F$ is called a {\sl set of defining relations} of $G$.
A minimal presentation \eqref{eq:presentation} induces an exact sequence in cohomology
\begin{equation}
 \begin{tikzpicture}[descr/.style={fill=white,inner sep=2pt}]
        \matrix (m) [
            matrix of math nodes,
            row sep=3em,
            column sep=3em,
            text height=1.5ex, text depth=0.25ex
        ]
        {    0 & \rmH^1(G,\Z/p) & \rmH^1(F,\Z/p)   &  \rmH^1(R,\Z/p)^F \\
               & \rmH^2(G,\Z/p) & \rmH^2(F,\Z/p)=0 &   \\
           };

        \path[overlay,->, font=\scriptsize,>=latex]
        (m-1-1) edge node[auto] {} (m-1-2) 
        (m-1-2) edge node[auto] {$\Inf_{G,F}^1$} (m-1-3) 
        (m-1-3) edge node[auto] {$\res_{F,R}^1$} (m-1-4) 
        (m-1-4) edge[out=355,in=175] node[descr,yshift=0.3ex] {$\mathrm{trg}$} (m-2-2)
        (m-2-2) edge node[auto] {$\Inf_{G,F}^2$} (m-2-3) ;
\end{tikzpicture}
        \end{equation}
where $\Inf_{G,F}^1$ is an isomorphism, and $\rmH^2(F,\Z/p)=0$ as $F$ is free.
Hence, also the map $\mathrm{trg}$ is an isomorphism.
Altogether, one has 
\begin{equation}\label{eq:H2}
\xymatrix{
\left(R/R^p[R,F]\right)^\ast\ar[r]^-{\sim} & \rmH^1(R,\Z/p)^F\ar[r]^-{\mathrm{trg}}& \rmH^2(G,\Z/p)}
\end{equation}
{(for the left-side isomorphism see, e.g., \cite[Ch.~I, \S~4.3]{serre:galc}).}
Therefore, since a set of defining relations of $G$ gives rise to a basis of the $\Z/p$-vector space $R/R^p[R,F]$, it yields a basis of $\rmH^2(G,\Z/p)$, via the isomorphism $\mathrm{trg}$, as well.

Let $F^{(3)}$ be the third term of the {\sl descending $p$-central series} of $F$, i.e.,  
$$F^{(3)}=\Phi(F)^p\cdot[\Phi(F),F]$$
(see, e.g., \cite[Def.~3.8.1]{nsw:cohn}).
Then the quotient $\Phi(F)/F^{(3)}$ is a $p$-elementary abelian pro-$p$ group, and thus it may be considered as a $\Z/p$-vector space.
If we consider --- with an abuse of notation --- $\calX=\{x_1,\ldots,x_d\}$ as a minimal generating set of $F$ too, then 
every element $r$ of $F'$ may be written as
 \[
  r=\prod_{1\leq i<j\leq d}[x_i,x_j]^{a_{ij}}\cdot y\qquad 
 \]
for some $a_{ij}\in\Z/p$ and $y\in F^{(3)}$, and the exponents $a_{ij}$ are uniquely determined \cite[Prop.~3.9.13--(i)]{nsw:cohn}.
Consequently, the set
\[
 \left\{\:[x_i,x_j]\cdot F^{(3)}\:\mid\:1\leq i<j\leq d\:\right\}
\]
is a linearly independent subset of $\Phi(F)/F^{(3)}$.  

Set $\mathcal{I}=\{1,\ldots,d\}$, and consider the set $\mathcal{I}\times\mathcal{I}$ endowed with the lexicographic order $\prec$ inherited from $\mathcal{I}$, i.e., $(i,j)\prec(i',j')$ if $i<i'$ or $i=i'$ and $j< j'$.
The following  result relates elementary commutators and cup-products (cf. \cite[Prop.~3.9.13--(ii)]{nsw:cohn}).

\begin{prop}\label{prop:cupprod}
Let $G$ be a finitely generated pro-$p$ group with minimal generating set $\calX=\{x_1,\ldots,x_d\}$, let \eqref{eq:presentation} be a minimal presentation of $G$, and let $\calX^\ast=\{\chi_1,\ldots,\chi_d\}$ be the basis of $\rmH^1(G,\Z/p)$ dual to $\mathcal{X}$.
Suppose that $\{r_1,\ldots,r_m\}$ is a subset of $R$ such that
\[\begin{split}
   r_1 &=\left[ x_{i(1)},x_{j(1)}\right]\prod_{\substack{1\leq i<j\leq d \\ (i,j)\succ(i(1),j(1))}}[x_i,x_j]^{a(1)_{i,j}} y_1,\\
   &\;\vdots \\
   r_m&=\left[ x_{i(m)},x_{j(m)}\right]\prod_{\substack{1\leq i<j\leq d \\ (i,j)\succ(i(m),j(m))}}[x_i,x_j]^{a(m)_{i,j}} y_m,
  \end{split}\]
for some $y_1,\ldots,y_m\in F^{(3)}$, where $a(h)_{i,j}\in\Z/p$ for every $h=1,\ldots,m$ and $1\leq i<j\leq d$, and $ (i(1),j(1))\prec \ldots \prec (i(m),j(m))$.
Then $\{\chi_{i(1)}\smallsmile\chi_{j(1)},\:\ldots,\:\chi_{i(m)}\smallsmile\chi_{j(m)}\}$ is a linearly independent subset of $\rmH^2(G,\Z/p)$.
\end{prop}



Observe that given a set of relations $\{r_1,\ldots,r_d\}\subseteq R\cap F'$ such that their images in the quotient $RF^{(3)}/F^{(3)}$ form a linearly independent subset, one may always assume that they satisfy the properties described in Proposition~\ref{prop:cupprod}--(ii), after performing Gau{\ss} reduction (cf. \cite[Rem.~2.5]{qsv:quadratic}).


Let $H_1,H_2$ be subgroups of a pro-$p$ group $G$ such that $H_1\supseteq H_2$.
Henceforth, for $\alpha\in\rmH^1(H_1,\Z/p)$, $\alpha|_{H_2}\in\rmH^1(H_2,\Z/p)$ will denote the restriction of $\alpha$ to $H_2$, while $\rmr_{H_1,H_2}\colon \rmH^2(H_1,\Z/p)\to\rmH^2(H_2,\Z/p)$ will denote the restriction map in degree 2. 
Recall that for every $\alpha_1,\alpha_2\in\rmH^1(H_1,\Z/p)$, one has
\begin{equation}\label{eq:cup res}
 \rmr_{H_1,H_2}(\alpha_1\smallsmile\alpha_2)=(\alpha_1\vert_{H_2})\smallsmile(\alpha_2\vert_{H_2})
\end{equation}
(cf. \cite[Prop.~1.5.3]{nsw:cohn}).
Moreover, for $\alpha'\in\rmH^1(G,\Z/p)$ and $V,W$ subspaces of $\rmH^1(G,\Z/p)$, we set
 \[
  \begin{split}
   \alpha'\smallsmile V &:= \left\{\alpha'\smallsmile\beta\:\mid\:\beta\in V\right\}, \\
   V\smallsmile W &:= \left\{\beta\smallsmile\beta'\:\mid\:\beta\in V,\:\beta'\in W\right\}. \\
 \end{split}
 \]


\subsection{Graphs and RAAGs}\label{ssec:graphs}
For the definition of simplicial graph we follow \cite[\S~1.1]{graph:book}.
A {\sl simplicial graph} is a pair $\Gamma = (\calV, \calE)$ of sets such that $\calE \subseteq [\calV]^2$, i.e., the elements
of $\calE$ are 2-element subsets of $\calV$ (we always assume implicitly that $\calV\cap\calE=\varnothing$). 
The elements of $\calV$ are the {\sl vertices} of $\Gamma$, the elements of $\calE$ are its {\sl edges}.
One may realize geometrically a simplicial graph by drawing a dot for each vertex and joining two of these dots by a line if the corresponding two vertices form an edge.
Henceforth, we will always deal with finite simplicial graphs, i.e., with a finite number of vertices.

\begin{rem}\rm
In \cite{serre:trees}, a simplicial graph is called an {\sl unoriented combinatorial graph}.
\end{rem}

\begin{defin}\rm
 Let $\Gamma=(\calV,\calE)$ be a simplicial graph.
 \begin{itemize}
  \item[(a)] $\Gamma$ is said to be {\sl complete} if $\calE=[\calV]^2$, i.e., every vertex is joined to any other vertex.
  \item[(b)] A simplicial graph $\Gamma'=(\calV',\calE')$ is a {\sl subgraph} of $\Gamma$ if $\calV'\subseteq \calV$ and $\calE'\subseteq \calE$; $\Gamma'$ is a {\sl proper} subgraph if $\Gamma'\neq\Gamma$; finally, $\Gamma'$ is said to be an {\sl induced} subgraph if in addition $\calE'=\calE\cap[\calV']^2$.
  \item[(c)] An induced subgraph $\Gamma'=(\calV',\calE')$ of $\Gamma$ is called an {\sl $n$-clique of $\Gamma$} if $\Gamma'$ is a complete simplicial graph with $n$ vertices; while $\Gamma'$ is called an {\sl $n$-cycle of $\Gamma$}, with $n\geq 3$, if $\Gamma'$ is a cycle with $n$ vertices, i.e., $\calV'=\{v_1,\ldots,v_n\}$ and
  \[\calE'=\left\{\:\{v_1,v_2\},\:\{v_2,v_3\},\:\ldots,\:\{v_{n-1},v_n\},\:\{v_n,v_1\}\:\right\}.\]
  \item[(d)] $\Gamma$ is said to be the {\sl pasting} of two proper induced subgraphs $\Gamma_1=(\calV_1,\calE_1)$ and $\Gamma_2=(\calV_2,\Gamma_2)$ {\sl along a common induced subgraph} $\Gamma'=(\calV',\calE')$ if $\calV=\calV_1\cup\calV_2$ and $\calV'=\calV_1\cap\calV_2$ --- namely, $\Gamma$ is the ``union'' of $\Gamma_1$ and $\Gamma_2$, and $\Gamma'$ is the ``intersection'' of $\Gamma_1$ and $\Gamma_2$.
 \end{itemize}
\end{defin}

{Given a simplicial graph $\Gamma=(\calV,\calE)$, we set $\dd(\Gamma)=|\calV|$, and we denote the number of connected components of $\Gamma$ by $\rmr(\Gamma)$.}

\begin{exam}\label{exam:graph}\rm
 Set $\calV=\{v_1,\ldots,v_5\}$ and 
 $$\calE=\left\{\:\{v_1,v_2\},\:\{v_1,v_3\},\:\{v_1,v_4\},\:\{v_1,v_5\},\:\{v_2,v_3\},\:\{v_3,v_4\},\:\{v_4,v_5\}\:\right\}.$$
 The simplicial graph $\Gamma=(\calV,\calE)$ has geometric realization
 
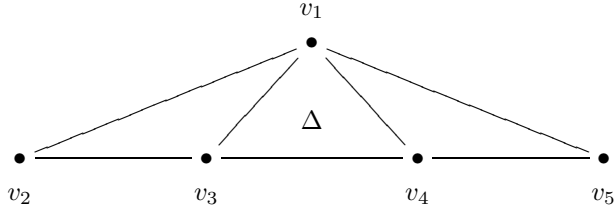
\[ 
 \xymatrix@R=1.5pt{  &&& v_1 &&& \\ &&& \bullet\ar@{-}[ddddlll]\ar@{-}[ldddd]\ar@{-}[rdddd]\ar@{-}[rrrdddd] & && \\
  \\ \\ &&&\Delta&&&  \\ 
 \bullet\ar@{-}[rr] && \bullet\ar@{-}[rr] && \bullet\ar@{-}[rr] && \bullet
 \\  v_2 && v_3 && v_4 && v_5}
 \]
and it is the pasting of the two induced subgraphs $\Gamma_1=(\calV_1,\calE_1)$ and $\Gamma_2=(\calV_2,\calE_2)$, with $\calV_1=\calV\smallsetminus\{v_5\}$ and $\calV_2=\calV\smallsetminus\{v_2\}$, along the common subgraph $\Delta$, which is the triangle with vertices $v_1,v_3,v_4$.
Moreover, if $\Delta'$ and $\Delta''$ are the triangles with vertices $v_1,v_2,v_3$ and $v_1,v_4,v_5$ respectively, then $\Delta,\Delta',\Delta''$ are the 3-cliques --- and the $3$-cycles --- of $\Gamma$, which has no $n$-cliques nor $n$-cycles for $n>3$.
\end{exam}

As mentioned in the Introduction, a simplicial graph $\Gamma=(\calV,\calE)$ is said to be {\sl chordal} (or {\sl triangulated}) if it has no cycles with more than 3 vertices --- e.g., the simplicial graph in Example~\ref{exam:graph} is chordal.
Clearly, this property is hereditary, namely, every induced subgraph of a chordal simplicial graph is again chordal. 
One has the following characterization of chordal simplicial graphs (cf., e.g., \cite[Prop.~5.5.1]{graph:book}).

\begin{prop}\label{prop:chordal}
A simplicial graph is chordal if, and only if, it can be constructed recursively by pasting along complete subgraphs (i.e., cliques), starting from complete simplicial graphs.
\end{prop}

\begin{exam}\rm
 If $\Gamma$ is as in Example~\ref{exam:graph}, then $\Gamma$ is the pasting of $\Gamma_1$ and $\Gamma_2$ along the complete simplicial graph $\Delta$, and in turn $\Gamma_1$ is the pasting of $\Delta'$ and $\Delta$ along the subgraph with vertices $v_1,v_3$ and edge $\{v_1,v_3\}$, which is complete --- and analogously $\Gamma_2$.
On the other hand, $\Gamma$ is not of elementary type, as the induced subgraph of $\Gamma$ with vertices $v_2,v_3,v_4,v_5$ is the graph $\mathrm{L}_3$.
\end{exam}

Given a simplicial graph $\Gamma=(\calV,\calE)$, the associated pro-$p$ RAAG is the pro-$p$ group $G_\Gamma$ with presentation
\begin{equation}\label{eq:def RAAG}
  G_{\Gamma}=\left\langle\:v\in\calV\:\mid\:[v,w]=1\:\text{for all }\{v,w\}\in\calE\:\right\rangle.
\end{equation}
The next result summarizes some useful and well-known features of pro-$p$ RAAGs, which follow from the definition of a pro-$p$ RAAG \eqref{eq:def RAAG} (and from the recursive procedure to construct a chordal simplicial graph, cf. Proposition~\ref{prop:chordal}, in the case of item~(iv)).

\begin{prop}
 Let $\Gamma=(\calV,\calE)$ be a simplicial graph, and let $G_\Gamma$ be the associated pro-$p$ RAAG.
 \begin{itemize}
 {\item[(i)] The pro-$p$ group $G_\Gamma$ is torsion-free, and the abelianization $G_\Gamma/G_\Gamma'$}
 is isomorphic, as an abelian pro-$p$ group, to the free $\Z_p$-module generated by $\calV$ --- and thus also to $\Z_p^{\dd(\Gamma)}$.
 \item[(ii)] For every induced subgraph $\Gamma'=(\calV',\calE)$ of $\Gamma$, the inclusion $\calV'\hookrightarrow\calV$ induces an isomorphism of pro-$p$ groups from $G_{\Gamma'}$ to the subgroup of $G_\Gamma$ generated by $\calV'$.
  \item[(iii)] If $\Gamma_1,\ldots,\Gamma_r$ are the connected components of $\Gamma$, then 
  $G_\Gamma\simeq G_{\Gamma_1}\amalg\ldots\amalg G_{\Gamma_r}$ {\rm(}here $\amalg$ denotes the free pro-$p$ product of pro-$p$ groups, cf. \cite[\S~9.1]{ribzal:book}{\rm)}.
  \item[(iv)] If $\Gamma$ is chordal and connected, then $G_\Gamma$ can be constructed recursively via proper amalgamated free pro-$p$ products over free abelian subgroups, starting from free abelian pro-$p$ groups
  {\rm(}for the definition of proper amalgamated free pro-$p$ product see \cite[\S~9.2]{ribzal:book}{\rm)}.
 \end{itemize}
\end{prop}

{ 
\begin{exam}\label{ex:GGamma chordal}\rm
 Let $\Gamma$ be as in Example~\ref{exam:graph}, and let $G_\Gamma$ be the associated pro-$p$ RAAG.
 Then one has the decomposition as proper amalgamated free pro-$p$ product
\begin{equation}\label{eq:GGamma chordal 2}
 G_\Gamma= G_{\Gamma_1}\amalg_{G_\Delta} G_{\Gamma_2} 
\end{equation}
(here $\Gamma_1,\Gamma_2$ are as in Example~\ref{exam:graph}), if one considers $\Gamma$ as the patching of $\Gamma_1$ and $\Gamma_2$ along $\Delta$.
Equivalently,
 \begin{equation}\label{eq:GGamma chordal 1}
G_\Gamma = \underbrace{\left(G_{\Delta'}\amalg_{A_1} G_\Delta\right)}_{G_{\Gamma_1}}\amalg_{A_2}G_{\Delta''}  
 \end{equation}
(here $A_1$ and $A_2$ are the subgroups of $G_\Gamma$ generated by $v_1,v_3$ and $v_1,v_4$ respectively).
 Observe that $A_1\simeq A_2\simeq\Z_p^2$, while $G_{\Delta'}\simeq G_\Delta\simeq G_{\Delta''}\simeq\Z_p^3$.
\end{exam}}


\subsection{The cohomology of pro-$p$ RAAGs}\label{ssec:raags cohom}

Given a simplicial graph $\Gamma=(\mathcal V,\mathcal E)$, with  $\calV=\{v_1,\ldots,v_d\}$, let $V$ denote the $\Z/p$-vector space generated by $\calV$.
One defines the {\sl Stanley-Reisner $\Z/p$-algebra} $\mathbf{\Lambda}_\bullet(\Gamma)$ associated to $\Gamma$ as the quotient 
\[
 \bfLam_\bullet(\Gamma)=\coprod_{n\geq0}\Lambda_n(\Gamma)=\frac{\bfLam_{\bullet}(V)}{\left(\:v\wedge w\:\mid \:\{v,w\}\notin\calE\:\right)}
\]
of the exterior algebra $\mathbf{\Lambda}_\bullet(V)=\coprod_{n\geq0}\Lambda_n(V)$ generated by $V$ over the ideal generated by the wedge products of disjoint vertices.
The algebra $\bfLam_\bullet(\Gamma)$ inherits the grading from the exterior algebra $\bfLam_\bullet(V)$, and hence it is a non-negatively graded connected algebra of finite type (i.e., $\Lambda_0(\Gamma)=\Z/p$ and $\dim(\Lambda_n(\Gamma))<\infty$ for every $n\geq0$).
In particular, $\bfLam_\bullet(\Gamma)$ is a {\sl quadratic algebra}, as $v\wedge w\in\Lambda_2(V)$ for every $v,w\in\calV$ (for the definition of quadratic algebra see, e.g., \cite[\S~1]{qsv:quadratic}).

{For $v_{i_1},\ldots, v_{i_n}\in\calV$, let $v_{i_1}\cdots v_{i_n}$ denote the image of $v_{i_1}\wedge\cdots\wedge v_{i_n}\in\Lambda_n(V)$ in $\Lambda_n(\Gamma)$.
The kernel of the epimorphism $\psi_n\colon\Lambda_n(V)\twoheadrightarrow\Lambda_n(\Gamma)$ has a basis
\[
 \left\{\:v_{i_1}\wedge\cdots\wedge v_{i_n} \:\mid\:1\leq i_1<\ldots<i_n\leq d,
 \:\left\{v_{i_s},v_{i_t}\right\}\notin\calE\text{ for some } s<t\:\right\}.
\]
On the other hand, one has $\psi_n(v_{i_1}\wedge\cdots\wedge v_{i_n})\neq0$ if, and only if, $\{v_{i_s},v_{i_t}\}\in\calE$ for every $1\leq s<t\leq n$ --- namely, if, and only if, there exists an $n$-clique $\Delta$ of $\Gamma$ such that $\calV(\Delta)=\{v_{i_1},\ldots,v_{i_n}\}$. 
Hence, for each positive degree $n$ the $\Z/p$-vector subspace $\Lambda_n(\Gamma)$ comes endowed with a basis 
\[ \left\{\:v_{i_1}\cdots v_{i_n}\:\mid\:1\leq i_1<\ldots<i_n\leq d\text{ and }\{\:v_{i_1},\:\ldots,\:v_{i_n}\:\}=\calV(\Delta)\:\right\}
\]
where $\Delta$ runs thorugh all $n$-cliques of $\Gamma$. }

Now let $G_\Gamma$ be the pro-$p$ RAAG associated to $\Gamma$.
By \eqref{eq:H1}, $\rmH^1(G_\Gamma,\Z/p)$ has a basis $\calV^\ast=\{\chi_1,\ldots,\chi_d\}$ dual to $\calV$.
Let $\Gamma^\ast=(\calV^\ast,\calE(\Gamma^\ast))$ be the simplicial graph with $\calE(\Gamma^\ast)=\{\{\chi_i,\chi_j\}\:\mid\:\{v_i,v_j\}\in\calE\}$.
Then
\[
 \bfH^\bullet(G_\Gamma)=\bfLam_\bullet(\Gamma^\ast)
\]
(cf., e.g., \cite[\S~3.2]{papa:raags} and \cite[Thm.~5.1]{AC:RAAGs EK}).
In particular, the $\Z/p$-cohomology algebra of a pro-$p$ RAAG is quadratic.
We will use $${ \mathcal{B}_{\Gamma^\ast}^2}:=\left\{\:\chi_i\smallsmile\chi_j\:\mid\:1\leq i<j\leq d\text{ and }\{v_i,v_j\}\in\calE\:\right\}$$
as the canonical basis of $\Lambda_2(\Gamma^\ast)=\rmH^2(G_\Gamma,\Z/p)$.


\subsection{Pro-$p$ RAAGs and maximal pro-$p$ Galois group}\label{ssec:Gal}

The following notions were introduced respectively in \cite{cq:bk} and in \cite{eq:kummer,qw:cyc}.

\begin{defin}\label{defin:BK 1cyc}\rm
Let $G$ be a pro-$p$ group.
 \begin{itemize}
  \item[(i)] $G$ is said to be a {\sl Bloch-Kato pro-$p$ group} if for every subgroup $H\subseteq G$, the $\Z/p$-cohomology algebra $\bfH^\bullet(H)$ is a quadratic $\Z/p$-algebra.
  \item[(ii)] $G$ is said to be {\sl 1-cyclotomic} if there exists a continuous $G$-module $M$, isomorphic to $\Z_p$ as an abelian pro-$p$ group, such that for every subgroup $H\subseteq G$ and for every positive integer $n$ the natural map \[
\rmH^1(H,M/p^nM)\longrightarrow \rmH^1(H,M/pM),
\]
induced by the epimorphism of continuous $G$-modules $M/p^nM\twoheadrightarrow M/pM$, is surjective.
 \end{itemize}
\end{defin}

\begin{rem}\rm Let $G$ be a pro-$p$ group.
\begin{itemize}
  \item[(a)] If $G$ is Bloch-Kato, then one has that $\rmH^2(H,\Z/p)=\rmH^1(H,\Z/p)\smallsmile\rmH^1(H,\Z/p)$ for every subgroup $H$ of $G$.
  \item[(b)] If $G$ is 1-cyclotomic and the action on the associated $G$-module $M$ is trivial, then $G$ is {\sl absolutely torsion-free}, i.e., $H/H'$ is a free abelian pro-$p$ group for every subgroup $H$ of $G$ (cf. \cite[Rem.~2.3]{cq:Galfeat}).
  Absolutely torsion-free pro-$p$ groups were introduced by T.~W\"urfel in \cite{wurf}.
 \end{itemize}
\end{rem}

If $\K$ is a field containing a root of 1 of order $p$, then the maximal pro-$p$ Galois group $G_{\K}(p)$ is both Bloch-Kato and 1-cyclotomic, respectively by the Norm Residue Theorem and by Kummer theory (cf. \cite[\S~4]{eq:kummer} and \cite[Thm.~1.1]{qw:cyc}).
In fact, by an earlier result of {A.S.~Merkur\cprime ev} and A.A.~Suslin --- which is the ``degree 2 version'' of the Norm Residue Theorem ---, one knows that 
$$\rmH^2(G_{\K}(p),\Z/p)=\rmH^1(G_{\K}(p),\Z/p)\smallsmile\rmH^1(G_{\K}(p),\Z/p)$$
(cf. \cite{mersus}, see also \cite[Thm.~6.4.4]{nsw:cohn}).
Therefore, a pro-$p$ group $G$ which is not Bloch-Kato (or just with elements in $\rmH^2(G,\Z/p)$ which do not arise from cup-products), or which is not 1-cyclotomic, can not occur as the maximal pro-$p$ Galois group of a field containing a root of 1 of order $p$ --- and hence as an absolute Galois group.
These properties were employed by I.~Snopce and P.A.~Zalesski{\u{\i}} to determine which pro-$p$ RAAGs occur as maximal pro-$p$ Galois groups (cf. \cite[Thm.~1.2 and Thm.~1.5]{sz:raags}).

\begin{thm}\label{thm:ilirpavel}
 Let $\Gamma$ be a simplicial graph, and let $G_{\Gamma}$ be the associated pro-$p$ RAAG.
 The following are equivalent:
 \begin{itemize}
  \item[(i)] $\Gamma$ is of elementary type; 
  \item[(ii)] $G_\Gamma$ occurs as the maximal pro-$p$ Galois group of a field $\K$ containing a root of 1 of order $p$;
  \item[(ii\cprime)] $G_\Gamma$ occurs as the maximal pro-$p$ Galois group of a field $\K$ containing all roots of 1 of $p$-power order;
  \item[(iii)] $G_\Gamma$ is a Bloch-Kato pro-$p$ group;
  \item[(iii\cprime)] $\rmH^2(H,\Z/p)=\rmH^1(H,\Z/p)\smallsmile\rmH^1(H,\Z/p)$ for every subgroup $H$ of $G_\Gamma$;
  \item[(iv)] $G_\Gamma$ is 1-cyclotomic;
  \item[(iv\cprime)] $G_\Gamma$ is absolutely torsion-free.
  \item[(v)] Every finitely generated subgroup of $G_\Gamma$ is again a pro-$p$ RAAG.
 \end{itemize}
\end{thm}


\section{Massey products}\label{sec:massey}

\subsection{Massey products in Galois cohomology}\label{ssec:massey Gal}

Let $G$ be a pro-$p$ group.
For $n\geq2$, the {\sl $n$-fold Massey product} on $\rmH^1(G,\Z/p)$ is a multi-valued map
\[
 \underbrace{\rmH^1(G,\Z/p)\times \ldots\times \rmH^1(G,\Z/p)}_{n\text{ times}}\longrightarrow \rmH^2(G,\Z/p).
\]
Given an $n$-tuple (with $n\geq2$) $\alpha_1,\ldots,\alpha_n$ of elements of $\rmH^1(G,\Z/p)$ (with possibly $\alpha_i=\alpha_j$ for some $1\leq i<j\leq n$), the (possibly empty) subset of $\rmH^2(G,,\Z/p)$ which is the value of the $n$-fold Massey product associated to the $n$-tuple $\alpha_1,\ldots,\alpha_n$ is denoted by $\langle\alpha_1,\ldots,\alpha_n\rangle$.
If $n=2$, then the 2-fold Massey product coincides with the cup-product, i.e., for $\alpha_1,\alpha_2\in \rmH^1(G,\Z/p)$ one has 
\begin{equation}\label{eq:cup 2massey}
 \langle\alpha_1,\alpha_2\rangle=\{\alpha\smallsmile\alpha_2\}\subseteq \rmH^2(G,\Z/p).
\end{equation}
For further details on this operation in the profinite and Galois-theoretic context, we direct the reader to \cites{vogel,ido:massey,mt:massey}. 
In particular, the definition of $n$-fold Massey products in the $\Z/p$-cohomology of pro-$p$ groups may be found in \cite[Def.~2.1]{mt:massey}.
For the purposes of our investigation, the properties described below --- and, in particular, the group-theoretic characterizations given by Proposition~\ref{prop:masse unip} --- will be enough.

Given an $n$-tuple $\alpha_1,\ldots,\alpha_n$ of elements of $\rmH^1(G,\Z/p)$, $n\geq2$, the Massey product $\langle\alpha_1,\ldots,\alpha_n\rangle$ is said:
\begin{itemize}
 \item[(a)] to be {\sl defined}, if $\langle\alpha_1,\ldots,\alpha_n\rangle\neq\varnothing$;
 \item[(b)] to {\sl vanish}, if $0\in\langle\alpha_1,\ldots,\alpha_n\rangle$.
\end{itemize}
Moreover, the pro-$p$ group $G$ is said to satisfy the {\sl $n$-Massey vanishing property} (with respect to $\Z/p$) if every defined $n$-fold Massey product in $\bfH^\bullet(G)$ vanishes.

In the following proposition we collect some properties of Massey products (cf., e.g., \cite[\S~1.2]{vogel} and  \cite[\S~2]{mt:massey}).

\begin{prop}\label{prop:massey cup}
 Let $G$ be a pro-$p$ group and let $\alpha_1,\ldots,\alpha_n$ be an $n$-tuple of elements of $\rmH^1(G,\Z/p)$, with $n\geq3$.
 Suppose that the $n$-fold massey product $\langle\alpha_1,\ldots,\alpha_n\rangle$ is defined.
 \begin{itemize}
 \item[(i)] For every $a\in\Z/p$ and $i\in\{1,\ldots,n\}$ one has 
 \[
 \varnothing\neq a\cdot \langle\alpha_1,\ldots,\alpha_n\rangle\subseteq \langle\alpha_1,\ldots,a\alpha_i,\ldots,\alpha_n\rangle.
 \] 
 In particular, if $\alpha_i=0$ for some $i$, then $0\in\langle\alpha_1,\ldots,\alpha_n\rangle$.
\item[(ii)] For all $i=1,\ldots,n-1$ one has $\alpha_i\smallsmile\alpha_{i+1}=0$.
\item[(iii)] The set $\langle\alpha_1,\ldots,\alpha_n\rangle$ is closed under adding $\alpha_1\smallsmile\alpha'$ and $\alpha_n\smallsmile\alpha'$ for any $\alpha'\in\rmH^1(G,\Z/p)$.
 \end{itemize}
 \end{prop}
 
 Let $\K$ be a field containing a root of 1 of order $p$. 
 In \cite[Conj.~1.1]{mt:conj}, J.~Mina\v{c} and N.D.~T\^an conjectured that the maximal pro-$p$ Galois group $G_{\K}(p)$ has the $n$-Massey vanishing property with respect to $\Z/p$, for every $n\geq3$.
 This conjecture has been proved in the following cases:
\begin{itemize}
 \item[(a)] if $n=3$, by E.~Matzri in the preprint \cite{eli:massey} (replaced by the paper by I.~Efrat and E.~Matzri \cite{EM:massey});
 \item[(b)] for every $n\geq3$ if $\K$ is a local field, by J.~Mina\v{c} and N.D.~T\^an (cf. \cite[Thm.~7.1]{mt:docu}) --- in fact, in this case $G_{\K}(p)$ has the {\sl strong} $n$-Massey vanishing property for every $n\geq3$ (cf. \cite[Prop.~4.1]{JT:U4}). 
 \item[(c)] for every $n\geq3$ if $\K$ is a number field, by J.~Harpaz and O.~Wittenberg (cf. \cite{HW:massey}).
\end{itemize}

In \cite[\S~7]{mt:massey}, one may find some examples of pro-$p$ groups with defined and non-vanishing 3-fold Massey products, and hence which do not occur as maximal pro-$p$ Galois groups of fields containing a root of 1 of order $p$.

\begin{exam}\label{ex:MT}\rm
Let $G$ be the pro-$p$ group with minimal presentation
\[
 G=\langle \:x_1,\ldots,x_5\:\mid\: [[x_1,x_2],x_3][x_4,x_5]\:\rangle.
\]
Then there is a $3$-fold Massey product in $\bfH^\bullet(G)$ which is defined but does not vanishes (cf. \cite[Ex.~7.2]{mt:massey}).
Th.S.~Weigel and the third-named author\footnote{During the problem session of the conference ``New Trends Around Profinite Groups'' (Sept. 2021), the third-named author promised a bottle of Franciacorta wine to anyone who will prove this.} suspect that $G$ is not a Bloch-Kato pro-$p$ group, and $G$ is not 1-cyclotomic (cf. \cite[Rem.~3.7]{qw:cyc}).
\end{exam}

\begin{rem}\label{rem:cupdef}\rm
 Following \cite[Def.~4.5]{JT:U4}, one says that a pro-$p$ group $G$ has the {\sl cup-defining $n$-fold Massey product property}, with $n\geq3$, if the $n$-fold Massey product $\langle\alpha_1,\ldots,\alpha_n\rangle$, associated to an $n$-tuple $\alpha_1,\ldots,\alpha_n$ of elements of $\rmH^1(G,\Z/p)$, is defined whenever
 \[
  \alpha_1\smallsmile\alpha_2=\alpha_2\smallsmile\alpha_3=\ldots=\alpha_{n-1}\smallsmile\alpha_n=0.\]
Thus, $G$ has the strong $n$-fold Massey vanishing property if, and only if, it has both the $n$-fold Massey vanishing property and the cup-defining $n$-fold Massey product property.
Moreover, if $G$ has the cup-defining $n$-fold Massey product property, then $G$ has the vanishing $(n-1)$-fold Massey vanishing property, as observed in \cite[Rem.~4.6]{JT:U4}.
Therefore, $G$ has the strong $n$-fold Massey vanishing property for every $n\geq3$ if, and only if, it has the cup-defining $n$-fold Massey product property for every $n\geq3$.
In \cite[Question~4.2]{JT:U4} it is asked whether the maximal pro-$p$ Galois group of a field containing a root of 1 of order $p$ has the the strong $n$-fold Massey vanishing property for every $n\geq3$.
\end{rem}


\subsection{Massey products and unipotent representations}\label{ssec:unipotent}

Massey products for a pro-$p$ group $G$ may be translated in terms of unipotent upper-triangular representations of $G$ as follows.
For $n\geq 2$ let
\[
 \dbU_{n+1}=\left\{\left(\begin{array}{ccccc} 1 & a_{1,2} & \cdots & & a_{1,n+1} \\ & 1 & a_{2,3} &  \cdots & \\
 &&\ddots &\ddots& \vdots \\ &&&1& a_{n,n+1} \\ &&&&1 \end{array}\right)\mid a_{i,j}\in\Z/p \right\}\subseteq 
 \mathrm{GL}_{n+1}(\Z/p)
\]
be the group of unipotent upper-triangular $(n+1)\times(n+1)$-matrices over $\Z/p$.
Then $\dbU_{n+1}$ is a finite $p$-group.
Moreover, let $I_{n+1}$ and $E_{i,j}$ denote respectively the identity $(n+1)\times(n+1)$-matrix and the $(n+1)\times(n+1)$-matrix with 1 at entry $(i,j)$ and 0 elsewhere, for $1\leq i<j\leq n+1$.
The center of $\dbU_{n+1}$ is the subgroup
  $$\mathrm{Z}(\dbU_{n+1})=I_{n+1}+\Z/p\cdot E_{1,n+1}=\left\{\:I_{n+1}+a\cdot E_{1,n+1}\:\mid\: a\in\Z/p\:\right\}$$
(cf., e.g., \cite[\S~3]{mt:massey} or \cite[p.~308]{eq:kummer}).
Set $\bar\dbU_{n+1}=\dbU_{n+1}/\mathrm{Z}(\dbU_{n+1})$.
For a homomorphism of pro-$p$ groups $\rho\colon G\to\dbU_{n+1}$, respectively $\bar \rho\colon G\to\bar\dbU_{n+1}$, and for $1\leq i\leq n$, let $\rho_{i,i+1}$, resp. $\bar\rho_{i,i+1}$, denote the projection of $\rho$, resp. $\bar\rho$, on the $(i,i+1)$-entry.
Observe that $$\rho_{i,i+1}\colon G\longrightarrow\Z/p\qquad \text{and}\qquad \bar\rho_{i,i+1}\colon G\longrightarrow\Z/p$$ are homomorphisms of pro-$p$ groups, and thus we may consider $\rho_{i,i+1}$ and $\bar \rho_{i,i+1}$ as elements of $\rmH^1(G,\Z/p)$.
One has the following ``pro-$p$ translation'' of a result of W.~Dwyer which interpretes Massey product in terms of unipotent upper-triangular represetations (cf., e.g., \cite[Lemma~9.3]{eq:kummer}).

\begin{prop}\label{prop:masse unip}
Let $G$ be a pro-$p$ group and let $\alpha_1,\ldots,\alpha_n$ be an $n$-tuple of elements of $\rmH^1(G,\Z/p)$, with $n\geq2$.
\begin{itemize}
 \item[(i)] The $n$-fold Massey product $\langle\alpha_1,\ldots,\alpha_n\rangle$ is defined if and only if there exists a continuous homomorphism $\bar \rho\colon G\to\bar\dbU_{n+1}$ such that $\bar\rho_{i,i+1}=\alpha_i$ for every $i=1,\ldots,n$.
 \item[(ii)] The $n$-fold Massey product $\langle\alpha_1,\ldots,\alpha_n\rangle$ vanishes if and only if there exists a continuous homomorphism $ \rho\colon G\to\dbU_{n+1}$ such that $\rho_{i,i+1}=\alpha_i$ for every $i=1,\ldots,n$.
\end{itemize}
\end{prop}


\subsection{Proof of Theorem~1.1}\label{ssec:massey raags}

We are ready to prove Theorem~\ref{thm:massey intro}.

\begin{proof}[Proof of Theorem~1.1]
Let $\alpha_1,\ldots,\alpha_n$ a sequence of elements of $\rmH^1(G_\Gamma,\Z/p)$ such that $\alpha_h\smallsmile\alpha_{h+1}=0$ for every $h=1,\ldots,n$ --- possibly, $\alpha_h=\alpha_{h'}$ for some $h'\neq h$ ---, and write 
$ \alpha_h=a_{1,h}\chi_1+\ldots+a_{d,h}\chi_d$ for every $h$ (by Proposition~\ref{prop:massey cup}--(i) we may assume that $\alpha_h\neq0$ for every $h$).
Then for every $h=1,\ldots,d-1$
\begin{equation}\label{eq:sum massey}
\begin{split}
  \alpha_h\smallsmile\alpha_{h+1} &= \left(\sum_{i=1}^d a_{i,h}\chi_i\right)\smallsmile\left(\sum_{j=1}^d a_{j,h+1}\chi_j\right)\\
  &= \sum_{\substack{ 1\leq l<l'\leq d \\ \{v_l,v_{l'}\}\in\calE }}
  \left(a_{l,h}a_{l',h+1}-a_{l',h}a_{l,h+1}\right)\chi_l\smallsmile\chi_{l'} 
\end{split}
\end{equation}
--- recall that $\chi_j\smallsmile\chi_i=-\chi_i\smallsmile\chi_j$ for every $1\leq i,j\leq d$ 
---, which is trivial by hypothesis. Since ${ \mathcal{B}_{\Gamma^\ast}^2}$ is a basis of $\Lambda_2(\Gamma^\ast)$, one has
$$a_{l,h}a_{l',h+1}-a_{l',h}a_{l,h+1}=0$$ whenever $\{v_l,v_{l'}\}\in\calE$.

Let $F$ be the free pro-$p$ group generated by $\calV$, and let $\tilde\rho\colon F\to\dbU_{n+1}$ be the homomorphism of pro-$p$ groups defined by 
\[
 \tilde\rho(v_i)=\left(\begin{array}{cccccc} 1 & \alpha_1(v_i) & 0 & \cdots & & 0 \\
                                             & 1 & \alpha_2(v_i) & \cdots & & 0 \\
                                             &  & 1             & \ddots & & \vdots \\
                                             &  &  & \ddots &\alpha_{n-1}(v_i) & 0 \\
                                             &  &  &  & 1 & \alpha_n(v_i) \\
                                             &  &  &  &   & 1                 \end{array}\right).
\]
Observe that $\alpha_h(v_i)=a_{i,h}$ for every $h=1,\ldots,n$ and $i=1,\ldots,d$.
Then 
\[
 \tilde\rho(v_iv_j)=\left(\begin{array}{cccccc} 1 & a_{i,1}+a_{j,1} & a_{i,1}a_{j,2} & 0 & \cdots & 0 \\
                                             & 1 & a_{i,2}+a_{j,2} & a_{i,2}a_{j,3} & \cdots  & 0  \\
                                             &  & 1       & a_{i,3}+a_{j,3}      & \ddots &  \vdots \\
                                             &  &  & \ddots & \ddots & a_{i,n-1}a_{j,n} \\
                                             &  &  &  & 1 & a_{i,n}+a_{j,n} \\
                                             &  &  &  &  & 1                 \end{array}\right).
\]
If $\{v_i,v_j\}\in\calE$, then $a_{i,h}a_{j,h+1}=a_{j,h}a_{i,h+1}$, so that $\tilde\rho(v_iv_j)=\tilde\rho(v_jv_i)$; on the other hand, $v_i$ and $v_j$ commute in $G_\Gamma$.
Therefore, $\tilde\rho$ yields a homomorphism $\rho\colon G_\Gamma\to\dbU_{n+1}$ such that $\rho_{h,h+1}=\alpha_h$ for every $h=1,\ldots,n$, and by Proposition~\ref{prop:masse unip} the $n$-fold Massey product $\langle\alpha_1,\ldots,\alpha_n\rangle$ is defined and vanishes.
\end{proof}


\subsection{Pro-$p$ groups of $p$-absolute Galois type and Massey products}\label{ssec:pabsgaltype}

Let $G$ be a pro-$p$ group.
First of all, we underline that for every $\alpha\in\rmH^1(G,\Z/p)$, the sequence \eqref{eq:exactsequence intro} is a complex, i.e., 
\begin{equation}\label{eq:complex}
\mathrm{cor}_{N,G}^1(\alpha')\smallsmile\alpha=0\qquad\text{and}\qquad\rmr_{G,N}(\alpha''\smallsmile\alpha)=0
\end{equation}
for every $\alpha'\in\rmH^1(N,\Z/p)$ and $\alpha''\in\rmH^1(G,\Z/p)$.
Moreover, observe that if $\alpha\in\rmH^1(G,\Z/p)$ is equal to 0, then the sequence \eqref{eq:exactsequence intro} is trivially exact at both $\rmH^1(G,\Z/p)$ and $\rmH^2(G,\Z/p)$, as both $\mathrm{cor}_{N,G}^1$ and $\res_{G,N}^2$ are the identity maps, and the cup-product by $\alpha$ is the trivial map.

The following notion was introduced in \cite[Def.~6.1.1]{LLSWW}.

\begin{defin}\label{defin:p-massey}\rm
 A pro-$p$ group $G$ has the {\sl $p$-cyclic Massey vanishing property} if for all $\alpha,\beta\in\rmH^1(G,\Z/p)$ such that $\alpha\smallsmile\beta=0$, the $p$-fold Massey product
 \[
  \langle\underbrace{\alpha,\ldots,\alpha}_{p-1\text{ times}},\beta\rangle
 \]
vanishes.
\end{defin}

Observe that if $p=2$ then every pro-$2$ group has, trivially, the $2$-cyclic vanishing property by \eqref{eq:cup 2massey}.

If $\K$ is a field containing a root of 1 of order $p$, then the maximal pro-$p$ Galois group $G_{\K}(p)$ has the $p$-cyclic Massey vanishing property (cf. \cite[Thm.~4.3]{sharifi:massey}, see also \cite[\S~6.1]{LLSWW}).
The following remarkable result was proved by Lam et al. (cf. \cite[Prop.~6.1.3--6.1.4 and Thm.~C]{LLSWW}).

\begin{thm}\label{prop:LLSWW}
 Let $G$ be a pro-$p$ group. 
 \begin{itemize}
  \item[(i)] If the sequence \eqref{eq:exactsequence intro} is exact at $\rmH^2(G,\Z/p)$ for every $\alpha\in\rmH^1(G,\Z/p)$, then it is exact also at $\rmH^1(G,\Z/p)$ for every $\alpha$, i.e., $G$ is of $p$-absolute Galois type.
  \item[(ii)] The sequence \eqref{eq:exactsequence intro} is exact at $\rmH^1(G,\Z/p)$ for every $\alpha\in\rmH^1(G,\Z/p)$ if, and only if, $G$ has the $p$-cyclic Massey vanishing property.
  \item[(iii)] If $G$ has the $p$-cyclic Massey vanishing property, then it has also the 3-Massey vanishing property.
 \end{itemize}
\end{thm}

\begin{exam}\rm
 Let $G$ be the pro-$p$ group of Example~\ref{ex:MT}.
 Then by Theorem~\ref{prop:LLSWW}--(iii) $G$ has not the $p$-cyclic Massey vanishing property, and thus by Theorem~\ref{prop:LLSWW}--(ii) $G$ is not of $p$-absolute Galois type.
\end{exam}



\section{Pro-$p$ RAAGs of $p$-absolute Galois type}\label{sec:raags pabsgaltype}


\subsection{Pro-$p$ RAAGs of $p$-absolute Galois type}\label{ssec:RAAGs pagt}

By Theorem~\ref{thm:massey intro}, the pro-$p$ RAAG $G_\Gamma$ associated to any simplicial graph $\Gamma=(\calV,\calE)$ has the $p$-cyclic Massey vanishing property.
Pick $\alpha\in\rmH^1(G_\Gamma,\Z/p)$, and set $N=\Ker(\alpha)$.
By Theorem~\ref{prop:LLSWW}--(ii), the sequence 
\begin{equation}\label{eq:exactsequence GGamma}
 \xymatrix@C=1.2truecm{\rmH^1(N,\Z/p)\ar[r]^-{\mathrm{cor}_{N,G_\Gamma}^1} & \rmH^1(G_\Gamma,\Z/p)\ar[r]^-{c_\alpha} &
 \rmH^2(G_\Gamma,\Z/p)\ar[r]^-{\rmr_{G_\Gamma,N}} & \rmH^2(N,\Z/p)},
\end{equation}
is exact at $\rmH^1(G_\Gamma,\Z/p)$ --- here $c_\alpha(\beta)=\beta\smallsmile\alpha$ for every $\beta\in\rmH^1(G_\Gamma,\Z/p)$.
Moreover, by Theorem~\ref{prop:LLSWW}--(i), in order to prove Theorem~\ref{thm:AbsGalType intro}--(i) it is enough to show that, given a chordal graph $\Gamma$, for every non-trivial $\alpha\in\rmH^1(G_\Gamma,\Z/p)$ the sequence \eqref{eq:exactsequence GGamma} is exact at $\rmH^2(G_\Gamma,\Z/p)$.


\subsection{Cup-product}\label{ssec:cup product}

{  
Set $\alpha=\chi_1+\ldots+\chi_m\in\rmH^1(G_\Gamma,\Z/p)$ for some $1\leq m\leq d$.
The goal of this subsection is to compute $\dim(\Img(c_\alpha))$.

Henceforth,} we identify $\rmH^n(G_\Gamma,\Z/p)=\Lambda_n(\Gamma^\ast)$.
Now, let $\Gamma_\alpha=(\calV(\Gamma_\alpha),\calE(\Gamma_\alpha))$ be the induced subgraph of $\Gamma$ with vertices $\calV(\Gamma_\alpha)=\{v_1,\ldots,v_m\}$ --- so that $m=\dd(\Gamma_\alpha)$ ---, and put $V_0=\Span\{\chi_{m+1},\ldots,\chi_d\}$.
Then 
\begin{equation}\label{eq:cohom Gamma sum}
\begin{split}
&\rmH^1(G_\Gamma,\Z/p)=\Lambda_1(\Gamma^\ast)=\Lambda_1(\Gamma_\alpha^\ast)\oplus V_0,\\
 &\rmH^2(G_\Gamma,\Z/p)=\Lambda_2(\Gamma^\ast)=\Lambda_2(\Gamma_\alpha^\ast)\oplus \left(\Lambda_1(\Gamma^\ast)\smallsmile V_0\right). 
\end{split}
\end{equation}
Consequently,
\begin{equation}\label{eq:img calpha}
\Img(c_\alpha)=\left(\Lambda_1(\Gamma_\alpha^\ast)\smallsmile\alpha\right)
\oplus \left( V_0\smallsmile\alpha\right), 
\end{equation}
where the left-side summand is a subspace of $\Lambda_2(\Gamma_\alpha^\ast)$, and the right-side summand is a subspace of $\Lambda_1(\Gamma^\ast)\smallsmile V_0$.
We study the dimension of these two summands in the next two propositions.

\begin{prop}\label{prop:cup alpha}
Let $\Gamma=(\calV,\calE)$ be a simplicial graph, with $\calV=\{v_1,\ldots,v_d\}$ --- so that $d=\dd(\Gamma)$ ---, and let $\calV^\ast=\{\chi_1,\ldots,\chi_d\}$ be the basis of $\Lambda_1(\Gamma^\ast)$ dual to $\calV$.
If $\alpha=\chi_1+\ldots+\chi_d$, then 
 \begin{equation}\label{eq:dim alpha cup alpha prop}
    \dim\left(\Lambda_1(\Gamma^\ast)\smallsmile\alpha\right) = \dd(\Gamma)-\rmr(\Gamma).
 \end{equation}
\end{prop}

\begin{proof}
Clearly, $\Lambda_1(\Gamma^\ast)\smallsmile\alpha$ is generated by the set $\mathcal{S}=\{\chi_1\smallsmile\alpha,\ldots,\chi_d\smallsmile\alpha\}$.
Moreover, 
\[
 0=\alpha\smallsmile\alpha=(\chi_1+\ldots+\chi_d)\smallsmile\alpha=\chi_1\smallsmile\alpha+\ldots+\chi_d\smallsmile\alpha,
\]
and thus for any $j\in\{1,\ldots,d\}$ one has $\chi_j\smallsmile\alpha=-\sum_{i\neq j}\chi_i\smallsmile\alpha$.
Therefore, for any $j$ the set 
\[
 \mathcal{S}_j=\left\{\:\chi_1\smallsmile\alpha,\:\ldots,\:\chi_{j-1}\smallsmile\alpha,\:
 \chi_{j+1}\smallsmile\alpha,\:\ldots,\:\chi_d\smallsmile\alpha\:\right\}
\]
is enough to generate $\Lambda_1(\Gamma^\ast)\smallsmile\alpha$.

Assume first that $\Gamma$ is a tree.
We proceed by induction on the number of vertices of $\Gamma$.
If $|\calV|<3$ then \eqref{eq:dim alpha cup alpha prop} holds trivially, so assume that $\Gamma$ has at least three vertices.
Up to renumbering the vertices, we may assume that $v_1$ is a leaf of $\Gamma$ (i.e., $\{v_1,v_i\}\in\calE$
only for one $i\in\{2,\ldots,d\}$, say $i=2$).
We claim that $\mathcal{S}_1$ is linearly independent.
Indeed, let $\beta$ be a $\Z/p$-linear combination of the elements of $\mathcal{S}_1$, i.e., 
\begin{equation}\label{eq:beta tree}
 \beta = a_2(\chi_2\smallsmile\alpha)+\ldots+a_d(\chi_d\smallsmile\alpha) =
 \sum_{\substack{\{v_i,v_j\}\in{ \calE} \\ { i<j} }}b_{ij}(\chi_i\smallsmile\chi_j)
\end{equation}
for some $a_i, b_{ij}\in\Z/p$ (where the $b_{ij}$'s are uniquely determined as ${ \mathcal{B}_{\Gamma^\ast}^2}$ is a basis of $\Lambda_2(\Gamma^\ast)$).
Since $v_1$ is a leaf, joined only to $v_2$, one has $b_{1,2}=-a_2$.
Thus, if $\beta=0$ then $b_{1,2}=0$, and hence $a_2=0$.
Now let $\rmT=(\calV(\rmT),\calE(\rmT))$ be the subtree of $\Gamma$ obtained removing the vertex $v_1$ --- namely, $\calV(\rmT)=\calV\smallsetminus\{v_1\}$ and $\calE(\rmT)=\calE\smallsetminus\{\{v_1,v_2\}\}$.
We denote the elements of $\calV(\rmT)^\ast$ by $\chi_2\vert_{\rmT},\ldots,\chi_d\vert_{\rmT}$, and we set $\alpha_{\rmT}=\chi_2\vert_{\rmT}+\ldots+\chi_d\vert_{\rmT}$.
By induction, 
$$\{\:(\chi_3\vert_{\rmT})\smallsmile\alpha_{\rmT},\ldots,(\chi_d\vert_{\rmT})\smallsmile\alpha_{\rmT}\:\}$$
is a linearly independent subset of $\Lambda_2(\rmT^\ast)$.
Let $\rmr_{\Gamma,\rmT}\colon\Lambda_2(\Gamma^\ast)\to\Lambda_2(\rmT^\ast)$ be the epimorphism of $\Z/p$-vector spaces defined by
\[
 \rmr_{\Gamma,\rmT}(\chi_i\smallsmile\chi_j)=\begin{cases}(\chi_i\vert_{\rmT})\smallsmile(\chi_j\vert_{\rmT}), & \text{if }i\neq1 \\0, &\text{if }i=1 \end{cases}\]
for $\chi_i\smallsmile\chi_j\in{ \mathcal{B}_{\Gamma^\ast}^2}$, $i<j$.
If $\beta=0$, then 
\[
 0=\rmr_{\Gamma,\rmT}(\beta)=a_3\left((\chi_3\vert_{\rmT})\smallsmile\alpha_{\rmT}\right)+
 \ldots+a_d\left((\chi_d\vert_{\rmT})\smallsmile\alpha_{\rmT}\right),
\]
and the inductive hypothesis implies that $a_3=\ldots=a_d=0$.
Therefore, $\mathcal{S}_1$ is linearly independent.

Assume now that $\Gamma$ is connected, and let $\rmT=(\calV(\rmT),\calE(\rmT))$ be a maximal tree of $\Gamma$. 
Then $\calV(\rmT)=\calV$ (cf. \cite[Ch.~I, \S~2.3, Prop.~11]{serre:trees}).
As above, we denote the elements of $\calV(\rmT)^\ast$ by $\chi_i\vert_{\rmT}$ for $1\leq i\leq d$, and $\alpha_{\rmT}=\chi_1\vert_{\rmT}+\ldots+\chi_d\vert_{\rmT}$.
Let $\rmr_{\Gamma,\rmT}\colon \Lambda_2(\Gamma^\ast)\to\Lambda_2(\rmT^\ast)$ be the epimorphism of $\Z/p$-vector spaces defined  by
\[
 \rmr_{\Gamma,\rmT}(\chi_i\smallsmile\chi_j)=\begin{cases}(\chi_i\vert_{\rmT})\smallsmile(\chi_j\vert_{\rmT}), & \text{if }\{v_i,v_j\}\in\calE(\rmT) \\0, &\text{if } \{v_i,v_j\}\notin\calE(\rmT),\end{cases}\]
and set 
\[
 \beta = a_2(\chi_2\smallsmile\alpha)+\ldots+a_d(\chi_d\smallsmile\alpha),\qquad\text{with }a_i\in\Z/p.
\]
Then 
\begin{equation}\label{eq:beta subtree}
 \rmr_{\Gamma,\rmT}(\beta)=a_2\left((\chi_2\vert_{\rmT})\smallsmile\alpha_{\rmT}\right)+
 \ldots+a_d\left((\chi_d\vert_{\rmT})\smallsmile\alpha_{\rmT}\right).
\end{equation}
Since $\{(\chi_i\vert_{\rmT})\smallsmile\alpha_{\rmT}\:\mid\:i=1,\ldots,d\}$ is a linearly independent subset of $\Lambda_2(\rmT^\ast)$, if $\beta=0$ then by \eqref{eq:beta subtree} one has $a_2=\ldots=a_d$, so that also $\mathcal{S}_1$ is linearly independent.

Finally, if $\Gamma_1,\ldots,\Gamma_{\rmr(\Gamma)}$ are the connected components of $\Gamma$, then 
$$ \Lambda_1(\Gamma^\ast)=\Lambda_1(\Gamma_1^\ast)\oplus\ldots\oplus\Lambda_1(\Gamma_{\rmr(\Gamma)}^\ast),$$
and moreover $\beta\smallsmile\beta'=0$ for $\beta\in\Lambda_1(\Gamma_i^\ast)$ and $\beta'\in\Lambda_1(\Gamma_j^\ast)$, $i\neq j$.
Hence, 
\[
 \Lambda_1(\Gamma^\ast)\smallsmile\alpha=
 \left(\Lambda_1(\Gamma_1^\ast)\smallsmile\alpha_1\right)\oplus\ldots\oplus\left(\Lambda_1(\Gamma_{\rmr(\Gamma)}^\ast)\smallsmile\alpha_{\rmr(\Gamma)}\right),
\]
where $\alpha_j=\sum_{v_i\in\calV(\Gamma_j)}\chi_i$ for every $j=1,\ldots,\rmr(\Gamma)$, and this yields \eqref{eq:dim alpha cup alpha prop}.
\end{proof}

\begin{prop}\label{prop:dim V0cup} 
Let
$$\calV_{0,\alpha}=\left\{\:v_j\:\mid\: m<j\leq d\text{ and }\{v_i,v_j\}\in\calE\text{ for some }1\leq i\leq m\:\right\}$$
be the set of vertices of $\Gamma$ not lying in $\Gamma_\alpha$ but joined to some vertices of $\Gamma_\alpha$.
Then
\begin{equation}\label{eq:dim V0cup}
 \dim(V_0\smallsmile\alpha)=|\calV_{0,\alpha}|.
\end{equation}
\end{prop}

\begin{proof}
Clearly, the set $\mathcal{S}_{0,\alpha}=\{\alpha\smallsmile\chi_j\mid v_j\in\calV_{0,\alpha}\}$ generates $V_0\smallsmile\alpha$.
On the other hand, we claim that $\mathcal{S}_{0,\alpha}$ is a linearly independent subset of $\Lambda_2(\Gamma^\ast)$.
Indeed, for $v_j\in\calV_{0,\alpha}$ one has 
\begin{equation}\label{eq:S0a}
\alpha\smallsmile \chi_j=\chi_1\smallsmile\chi_j+\ldots+\chi_m\smallsmile\chi_j,
\end{equation}
where at least one summand of the right-side term of \eqref{eq:S0a} is non-trivial, as $v_j\in\calV_{0,\alpha}$.
Moreover, observe that if $m<j'\leq d$, $j'\neq j$, and $v_{j'}\in\calV_{0,\alpha}$, then necessarily $\alpha\smallsmile\chi_{j'}\neq\alpha\smallsmile\chi_j$, as otherwise by \eqref{eq:S0a} one would have
\[
\chi_1\smallsmile\chi_j+\ldots+\chi_m\smallsmile\chi_j=\chi_1\smallsmile\chi_{j'}+\ldots+\chi_m\smallsmile\chi_{j'},
\]
and thus 
\begin{equation}\label{eq:repetitions}
\chi_i\smallsmile\chi_j=\chi_{i'}\smallsmile\chi_{j'}\qquad\text{for some }1\leq i,i'\leq m 
\end{equation}
such that $\chi_i\smallsmile\chi_j$ and $\chi_{i'}\smallsmile\chi_{j'}$ are not trivial, since $\mathcal{B}^{2}_{\Gamma^\ast}$ is a basis of $\Lambda_2(\Gamma^\ast)$.
But equality \eqref{eq:repetitions} is impossible, as $j'\neq i,j$ and hence $\{v_i,v_j\}\neq\{v_{i'},v_{j'}\}$.

For every $v_j\in\calV_{0,\alpha}$ set 
$$\mathcal{S}_{j,\alpha}=\left\{\:\chi_i\smallsmile\chi_j\:\mid\:1\leq i\leq m,\:\chi_i\smallsmile\chi_j\neq0\:\right\}.$$
Then, the sets $\mathcal{S}_{j,\alpha}$, with $v_j$ running through the elements of $\calV_{0,\alpha}$, are disjoint and non-empty subsets of the basis $\mathcal{B}_{\Gamma^\ast}^2$.
Therefore, by \eqref{eq:S0a} one has
$$\mathcal{S}_{0,\alpha}=\left\{\:\sum_{\beta\in\mathcal{S}_{j,\alpha}}\beta\:\mid\:v_j\in\calV_{0,\alpha}\:\right\}
\subseteq\bigoplus_{v_j\in\calV_{0,\alpha}}\Span\{\mathcal{S}_{j,\alpha}\},$$
so that $\mathcal{S}_{0,\alpha}$ is a linearly independent subset of $\Lambda_2(\Gamma^\ast)$, as claimed.
\end{proof}

Altogether, if $\alpha=\chi_1+\ldots+\chi_m\in\rmH^1(G_\Gamma,\Z/p)$ for some $1\leq m\leq d$, from \eqref{eq:img calpha} and from Propositions~\ref{prop:cup alpha}--\ref{prop:dim V0cup} one concludes that 
\begin{equation}\label{eq:dim Imcalpha}
\begin{split}
  \dim(\Img(c_\alpha)) &=\dim\left(\Lambda_1(\Gamma_\alpha^\ast)\smallsmile\alpha\right)+\dim(V_0\smallsmile\alpha) \\
  &=\left(\dd(\Gamma_\alpha)-\rmr(\Gamma_\alpha)\right)+|\calV_{0,\alpha}|,
\end{split}
\end{equation}
 where $\calV_{0,\alpha}$ is as in Proposition~\ref{prop:dim V0cup} (note that necessarily $\rmr(\Gamma_\alpha)\leq m$), with no restrictions on the {shape of the} simplicial graph $\Gamma$. 


\subsection{Restriction}\label{ssec:res}

{ 
The goal of this subsection is to study $\dim(\Img(\rmr_{G_\Gamma,N}))$ in case $\Gamma$ is a chordal graph.
By duality (cf. \ref{eq:H2}), this depends on how may defining relations of $G_\Gamma$ ``remain'' defining relations for a minimal presentation of $N$.

Throughout this subsection, we set $\alpha=\chi_1+\ldots+\chi_m\in\rmH^1(G_\Gamma,\Z/p)$ for some $1\leq m\leq d$, and we set 
$\Gamma_\alpha=(\calV(\Gamma_\alpha),\calE(\Gamma_\alpha))$, with $\calV(\Gamma_\alpha)=\{v_1,\ldots,v_m\}$, as in \S~\ref{ssec:cup product} (so that $m=\dd(\Gamma_\alpha)$).}
For $1\leq i\leq d$ set
\[
w_i =\begin{cases}
      v_iv_{i+1}^{-1}, & \text{if }i<m \\ v_i, & \text{if }i\geq m.
     \end{cases}\]
Then $\calW=\{w_1,\ldots,w_d\}$ is a minimal generating set of $G_\Gamma$, which induces a basis $\calW^\ast=\{\omega_1,\ldots,\omega_d\}$ of $\rmH^1(G_\Gamma,\Z/p)$ --- observe that $\omega_j=\chi_j$ for $j>m$. 
Set 
$$\calY_\alpha = \left\{\:w_1,\:\ldots,\:w_{m-1}\:\right\}\qquad\text{and}\qquad
\calY =\calY_\alpha\cup\left\{\:v_{m+1},\:\ldots,\:v_d\:\right\}=\calW\smallsetminus\{\:w_m\:\},$$
and let $H_\alpha$ and $H$ be the subgroups of $G_{\Gamma}$ generated respectively by $\calY_\alpha$ and $\calY$.
Since $\alpha(w_i)=0$ for $i\neq m$, these two subgroups are contained in $N$.

\begin{lem}\label{lem:Y generating}
 The sets $\calY$ and $\calY_\alpha$ are minimal generating sets of $H$ and $H_\alpha$ respectively.
\end{lem}

\begin{proof}
 Since $\calY$ is a subset of $\calW$, which is a minimal generating set of $G_\Gamma$, the cosets $w_i\Phi(G_\Gamma)$, with $i\neq m$, are linearly independent in the $\Z/p$-vector space $G_\Gamma/\Phi(G_\Gamma)$, and thus also the cosets $w_i\Phi(H)$, with $i\neq m$, are linearly independent in $H/\Phi(H)$ --- and analogously the cosets $w_i\Phi(H_\alpha)$, with $1\leq i<m$, are linearly independent in $H_\alpha/\Phi(H_\alpha)$.
\end{proof}

By Lemma~\ref{lem:Y generating}, and by duality, $\rmH^1(H,\Z/p)$ and $\rmH^1(H_\alpha,\Z/p)$ have bases 
$$\calY^\ast=\left\{\:\omega_i\vert_H\:\mid\:w_i\in\calY\:\right\}\qquad \text{and} \qquad
\calY_\alpha^\ast=\left\{\:\omega_i\vert_{H_\alpha}\:\mid\:w_i\in\calY_\alpha\:\right\}$$ respectively.
The next proposition gives a lower bound for the dimension of the image of the map $\rmr_{G_\Gamma,H_\alpha}$ restricted to the summand $\Lambda_2(\Gamma_\alpha^\ast)$ of $\rmH^2(G_\Gamma,\Z/p)$ (cf. \eqref{eq:cohom Gamma sum}).
The idea is the following.
If $v_1,v_2,v_3$ are vertices of $\Gamma_\alpha$, and they are joined to each other (namely, they are the vertices of a 3-clique of $\Gamma_\alpha$), then the commutator 
$$[w_1,w_2]=\left[v_1v_2^{-1},v_2v_3^{-1}\right]=[v_1,v_2][v_1,v_3^{-1}][v_2^{-1},v_3^{-1}]$$
is trivial: since $w_1,w_2\in\calY_\alpha$, the relation $[w_1,w_2]=1$ is a defining relation of $H_\alpha$ induced by the three defining relations $[v_1,v_2]=[v_1,v_3]=[v_2,v_3]=1$ of $G_\Gamma$.
We use this fact, together with the recursive procedure to construct a chordal graph via patching subgraphs along cliques, starting from cliques.

\begin{prop}\label{prop:res2}
If $\Gamma$ is a chordal simplicial graph, then 
 \begin{equation}\label{eq:eulerchar}
  \dim\left(\rmr_{G_\Gamma,H_\alpha}(\Lambda_2(\Gamma_\alpha^\ast))\right) \geq 
  |\calE(\Gamma_\alpha)|-\dd(\Gamma_\alpha)+\rmr(\Gamma_\alpha).
\end{equation}
\end{prop}

\begin{proof}
We proceed as follows: first, we suppose that $\Gamma_\alpha$ is complete, then we suppose that $\Gamma_\alpha$ is connected, and finally we deal with the general case.
 
 If $\Gamma_\alpha$ is complete, then $[w_i,w_j]=1$ for every $1\leq i<j\leq m-1$, and thus $H_\alpha\simeq\Z_p^{m-1}$.
From M.~Lazard's work \cite{lazard}, one knows that the $\Z/p$-cohomology algebra of a free abelian pro-$p$ group is an exterior algebra (see, e.g., \cite[Thm.~5.1.5]{sw:cohom}). 
Thus, $\rmH^\bullet(H_\alpha)$ is the exterior algebra generated by $\calY_\alpha^\ast$, and consequently
 \begin{equation}
  \dim\left(\rmH^2(H_\alpha,\Z/p)\right)=\binom{m-1}{2}=\binom{m}{2}-m+1.
 \end{equation}
Moreover, observe that $\rmH^2(H_\alpha,\Z/p)=\rmr_{G_\Gamma,H_\alpha}(\Lambda_2(\Gamma_\alpha^\ast))$, as 
the map $\res^1_{G_\Gamma,H_\alpha}$ is surjective and $\rmH^2(H_\alpha,\Z/p)=\Lambda_2(\rmH^1(H_\alpha,\Z/p))$.
This proves \eqref{eq:eulerchar} in this case.

Now suppose that $\Gamma_\alpha$ is connected and not complete.
We use the recursive procedure to construct a chordal graph (cf. Proposition~\ref{prop:chordal}): namely, we may find two proper induced subgraphs $\Gamma_1,\Gamma_2$ of $\Gamma_\alpha$ (which are chordal as well), whose intersection is a clique $\Delta$, such that $\Gamma_\alpha$ is the pasting of $\Gamma_1$ and $\Gamma_2$ along $\Delta$.
 Up to renumbering, we may suppose that 
 $$\calV(\Gamma_1)=\{\:v_1,\ldots,v_{m_1}\:\}\qquad\text{and}\qquad \calV(\Gamma_2)=\{\:v_{m_2},\ldots,v_m\:\},$$
with $1< m_2\leq m_1<m$, so that $\calV(\Delta)=\{v_{m_2},\ldots,v_{m_1}\}$.
 Let $H_1$, $H_2$ and $A$ be the subgroups of $H_\alpha$ generated, respectively, by $\{w_1,\ldots,w_{m_1-1}\}$, by $\{w_{m_2},\ldots,w_{m-1}\}$, and by $\{w_{m_2},\ldots,w_{m_1-1}\}$.
 Since $\Delta$ is complete, $A\simeq\Z_p^{m_2-m_1}$ --- so $A$ might be trivial, if $\Delta$ consists only of one vertex.
 Thus, 
 \begin{equation}\label{eq:HDelta H}  
 \rmH^2(A,\Z/p) = \Lambda_2(V_\Delta) = \rmr_{G_\Gamma,A}\left(\Lambda_2(\Gamma_\alpha^\ast)\right) \end{equation}
 where $V_\Delta=\Span\{\:\omega_i\vert_{A}\:\mid\: m_2\leq i<m_1\:\}$.
 Consider the sequence of $\Z/p$-vector spaces
 \begin{equation}\label{eq:sequence res amalg}
  \begin{tikzpicture}[descr/.style={fill=white,inner sep=2pt}]
        \matrix (m) [
            matrix of math nodes,
            row sep=3em,
            column sep=3em,
            text height=1.5ex, text depth=0.25ex
        ]
        {   \rmr_{G_\Gamma,H_\alpha}\left(\Lambda_2(\Gamma_\alpha^\ast)\right)   &  
        \rmr_{G_\Gamma,H_1}\left(\Lambda_2(\Gamma_\alpha^\ast)\right)\oplus\rmr_{G_\Gamma,H_2}\left(\Lambda_2(\Gamma_\alpha^\ast)\right) &
        \\
          &  \Lambda_2(V_\Delta)=\rmr_{G_\Gamma,A}\left(\Lambda_2(\Gamma_\alpha^\ast)\right)  &  0\;, \\
           };
        \path[overlay,->, font=\scriptsize,>=latex]
        (m-1-1) edge node[auto] {$f_1$} (m-1-2) 
        (m-1-2) edge[out=355,in=175] node[descr,yshift=0.3ex] {$f_2$} (m-2-2)
        (m-2-2) edge node[auto] {} (m-2-3) ;
\end{tikzpicture}
 \end{equation}
where $f_1=\rmr_{H_\alpha,H_1}\oplus\rmr_{H_\alpha,H_2}$ and $f_2=\rmr_{H_1,A}-\rmr_{H_2,A}$.
The map $f_2$ is surjective, as the subset $\{\:(\omega_i\vert_{H_k})\smallsmile(\omega_j\vert_{H_k})\:\mid\:m_2\leq i<j<m_1\:\}$ of $\rmH^2(H_k,\Z/p)$ is linearly independent for both $k=1,2$ --- as $[w_i,w_j]=1$ for every $m_2\leq i<j<m_1$ ---, and
$$\rmr_{H_k,A}\left((\omega_i\vert_{H_k})\smallsmile(\omega_j\vert_{H_k})  \right)=(\omega_i\vert_A)\smallsmile(\omega_j\vert_{A}).$$
Moreover, $\Img(f_1)\subseteq\Ker(f_2)$, as $\rmr_{H_k,A}\circ\rmr_{H_\alpha,H_k}=\rmr_{H_\alpha,A}$. 
Finally, let $\beta_1,\beta_2\in\Lambda_2(\Gamma_\alpha^\ast)$ be such that $(\rmr_{G_\Gamma,H_1}(\beta),\rmr_{G_\Gamma,H_2}(\beta_2))\in\Ker(f_2)$, and write
\[
 \beta_1=\sum_{1\leq i<j\leq m}a_{ij}(\omega_i\smallsmile\omega_j)\qquad\text{and}\qquad
 \beta_1=\sum_{1\leq i<j\leq m}b_{ij}(\omega_i\smallsmile\omega_j),
\]
with $a_{ij},b_{ij}\in\Z/p$ --- here we employ $\{\:\omega_1,\ldots,\omega_m\:\}$ as a basis of $\Lambda_1(\Gamma_\alpha^\ast)$.
Since $\rmr_{H_k,A}\circ\rmr_{G_\Gamma,H_k}=\rmr_{G_\Gamma,A}$ for both $k=1,2$, and $\omega_i\vert_A=0$ for $i<m_2$ or $i\geq m_1$, one has
\[
 \rmr_{H_1,A}\left(\rmr_{G_\Gamma,H_k}(\beta_1)\right)=
 \sum_{m_2\leq i<j<m_1}a_{ij}\left((\omega_i\vert_A)\smallsmile(\omega_j\vert_A)\right),
\]
and similarly for $\beta_2$, and therefore $a_{ij}=b_{ij}$ for $m_2\leq i<j<m_1$.
Set 
\[\beta=\sum_{1\leq i<j\leq m}c_{ij}(\omega_i\smallsmile\omega_j)\in\Lambda_2(\Gamma_\alpha^\ast)\qquad \text{with } 
c_{ij}=\begin{cases} a_{ij} &\text{for }j<m_1, \\ b_{ij} & \text{for }i\geq m_2.\end{cases}\]
Then $\rmr_{G_\Gamma,H_k}(\beta)=\rmr_{G_\Gamma,H_k}(\beta_k)$ for both $k=1,2$, and thus 
$(\rmr_{G_\Gamma,H_1}(\beta_1),\rmr_{G_\Gamma,H_2}(\beta_2))=f_1(\rmr_{G_\Gamma,H_\alpha}(\beta))$.

Altogether, the sequence \eqref{eq:sequence res amalg} is exact.
Moreover, observe that for both $k=1,2$ one has 
\[ \rmr_{G_\Gamma,H_k}\left(\Lambda_2(\Gamma_\alpha^\ast)\right)=\rmr_{G_\Gamma,H_k}\left(\Lambda_2(\Gamma_k^\ast)\right),\]
and by induction one has the inequality \eqref{eq:eulerchar} with $H_k$ instead of $H_\alpha$ and $\Gamma_k$ instead of $\Gamma_\alpha$, as $\Gamma_k$ is chordal, and a proper subgraph of $\Gamma$.
Therefore, from the exactness of \eqref{eq:sequence res amalg} and from \eqref{eq:HDelta H} we deduce that
\[\begin{split}
  \dim\left(\rmr_{G_\Gamma,H_\alpha}(\Lambda_2(\Gamma_\alpha^\ast))\right) \geq & 
  \dim\left(\rmr_{G_\Gamma,H_1}(\Lambda_2(\Gamma_\alpha^\ast))\right)+
  \dim\left(\rmr_{G_\Gamma,H_2}(\Lambda_2(\Gamma_\alpha^\ast))\right) - \binom{m_1-m_2}{2} \\
  \geq & \left[|\calE(\Gamma_1)|-(m_1-1)\right]+\left[|\calE(\Gamma_2)|-((m-m_2+1)-1)\right]
  \\& \qquad - \left[|\calE(\Delta)|-((m_1-m_2+1)-1)\right]\\
  = &\; |\calE(\Gamma_\alpha)|-(m-1),
  \end{split} \]
as $\calE(\Gamma_\alpha)=\calE(\Gamma_1)\cup\calE(\Gamma_2)$ and $\calE(\Delta)=\calE(\Gamma_1)\cap\calE(\Gamma_2)$, and $\dim(\Lambda_2(V_\Delta))=\binom{m_1-m_2}{2}$.
Thus, inequality \eqref{eq:eulerchar} holds for $\Gamma_\alpha$ chordal and connected (i.e., with $\rmr(\Gamma_\alpha)=1$).

Finally, let $\Gamma_{\alpha,1},\ldots,\Gamma_{\alpha,r}$ be the connected components of $\Gamma_\alpha$.
Then $$\Lambda_1(\Gamma_\alpha^\ast)=\Lambda_1(\Gamma_{\alpha,1}^\ast)\oplus\ldots\oplus\Lambda_1(\Gamma_{\alpha,r}^\ast).$$
Write $\alpha=\alpha_1+\ldots+\alpha_r$, where $\alpha_i\in \Lambda_1(\Gamma_{\alpha,i}^\ast)$ for each $i=1,\ldots,r$.
Since every connected component is disjoint to each other, one has $\beta\smallsmile\beta'=0$ for every $\beta\in\Lambda_1(\Gamma_{\alpha,i}^\ast)$ and $\beta'\in\Lambda_1(\Gamma_{\alpha,j}^\ast)$ with $1\leq i<j\leq r$, and therefore
\begin{equation}\label{eq:qui sotto}
 \Lambda_1(\Gamma_\alpha^\ast)\smallsmile\alpha=
 \left(\Lambda_1(\Gamma_{\alpha,1}^\ast)\smallsmile\alpha_1\right)\oplus\ldots\oplus
 \left(\Lambda_1(\Gamma_{\alpha,r}^\ast)\smallsmile\alpha_r\right).
\end{equation}
Since \eqref{eq:dim alpha cup alpha prop} holds for each connected component of $\Gamma_\alpha$, by \eqref{eq:qui sotto} it holds also for $\Gamma_\alpha$ itself.
This completes the proof.
 \end{proof}
 
 {  
 \begin{exam}\label{ex:res chordal}\rm
  Let $\Gamma$ be a simplicial chordal graph as in Example~\ref{exam:graph}, with associated pro-$p$ RAAG $G_\Gamma$.
  Set $\alpha=\chi_1+\ldots+\chi_5$ (so that $\Gamma_\alpha=\Gamma$), and let $H_\alpha\subseteq G_\Gamma$ and $\calY_\alpha$ be as above.
  Then in $H_\alpha$ one has the three relations
  \[\begin{split}
 [w_1,w_2]&=\left[v_1v_2^{-1},v_2v_3^{-1}\right]=1,\qquad \text{as }G_{\Delta'}=\langle\:v_1,v_2,v_3\:\rangle\simeq\Z_p^3,\\
 [w_1w_2,w_3]&=\left[ v_1v_3^{-1},v_3v_4^{-1}\right]=1,\qquad \text{as }G_{\Delta}=\langle\:v_1,v_3,v_4\:\rangle\simeq\Z_p^3,\\
 [w_1w_2w_3,w_4]&=\left[v_1v_4^{-1},v_4v_5^{-1}\right]=1,\qquad \text{as }G_{\Delta''}=\langle\:v_1,v_4,v_5\:\rangle\simeq\Z_p^3,
    \end{split}  \]
which are independent in the sense of Proposition~\ref{prop:cupprod}, and induced, respectively, by the relations of $G_{\Delta'}$, of $G_\Delta$, and of $G_{\Delta''}$.
Therefore, 
\[
 \dim\left(\rmr_{G_\Gamma,H_\alpha}(\Lambda_2(\Gamma^\ast))\right)\geq 3=|\calE|-\dd(\Gamma)+\rmr(\Gamma).
\]
In particular, by the proof of Proposition~\ref{prop:res2} the decomposition  \eqref{eq:GGamma chordal 2} yields
\[ \begin{split}
   \dim\left(\rmr_{G_\Gamma,H_\alpha}(\Lambda_2(\Gamma^\ast))\right)&\geq 
   (|\calE(\Gamma_1)|-\dd(\Gamma_1)+1)+(|\calE(\Gamma_2)|-\dd(\Gamma_2)+1) -\dim\left(\Lambda_2(V_\Delta)\right)\\
  &=(5-4+1)+(5-4+1)-\binom{2}{2}=3. 
 \end{split}\]
On the other hand, the subgroups $A_1,A_2$ of $G_\Gamma$ (cf. Example~\ref{ex:GGamma chordal}) induce no relations in $H_\alpha$, as 
$$A_1\cap H_\alpha=\langle\: w_1w_2\:\rangle\qquad\text{and}\qquad A_2\cap H_\alpha=\langle\: w_1w_2w_3\:\rangle,$$
and both are isomorphic to $\Z_p$.
Hence, by the proof of Proposition~\ref{prop:res2}, the decomposition \eqref{eq:GGamma chordal 1} yields
\[\begin{split}
   \dim\left(\rmr_{G_\Gamma,H_\alpha}(\Lambda_2(\Gamma^\ast))\right)&\geq 
   (|\calE(\Gamma_1)|-\dd(\Gamma_1)+1)+(|\calE(\Delta'')|-\dd(\Delta'')+1) -0\\
  &= \left((|\calE(\Delta')|-\dd(\Delta')+1)+(|\calE(\Delta)|-\dd(\Delta)+1)-0\right) + 1-0\\
   &= 1+1+1=3.
  \end{split}\]
\end{exam} }
 
The next proposition gives a lower bound for the dimension of the image of the map $\rmr_{G_\Gamma,H_\alpha}$ restricted to the summand $\Lambda_1(\Gamma^\ast)\smallsmile V_0$ of $\rmH^2(G_\Gamma,\Z/p)$ (cf. \eqref{eq:cohom Gamma sum}).
The idea is the following.
If $v_1,v_2\in\calV(\Gamma_\alpha)$ and $v_3\in\calV\smallsetminus\calV(\Gamma_\alpha)$, and $\{v_1,v_3\},\{v_2,v_3\}\in\calE$, then the commutator 
$$[w_1,w_3]=\left[v_1v_2^{-1},v_3\right]=[v_1,v_3][v_2^{-1},v_3]$$
is trivial: since $w_1,w_3\in\mathcal{Y}$, the relation $[w_1,w_3]=1$ is a defining relation of $H$ induced by the two defining relations $[v_1,v_3]=[v_2,v_3]=1$ of $G_\Gamma$.

\begin{prop}\label{prop:res 1}
Let $\Gamma$ be a simplicial graph.
Then 
\begin{equation}\label{eq:dim res 1}
\dim\left(\rmr_{G_\Gamma,H}(\Lambda_1(\Gamma^\ast)\smallsmile V_0)\right)
\geq |\calE_0|+\sum_{v_j\in\calV_{0,\alpha}}(e(v_j)-1),
  \end{equation}
  where $\calV_{0,\alpha}$ is defined as in Proposition~\ref{prop:dim V0cup}, $\calE_0=\{\{v_i,v_j\}\in\calE\:\mid\:m<i<j\leq d\}$, and $e(v_j)$ is the number of vertices of $\calV(\Gamma_\alpha)$ which are joined to $v_j$.
\end{prop}

\begin{proof}
 If $\{v_i,v_j\}\in\calE_0$, then $w_i=v_i$ and $w_j=v_j$, and thus one has the relation $[w_i,w_j]=1$ in $H$.

 On the other hand, for $v_j\in\calV_{0,\alpha}$, let $v_{j_1},\ldots,v_{j_{e(v_j)}}$ be the vertices lying in $\calV(\Gamma_\alpha)$ which are joined to $v_j$, with $1\leq j_1<\ldots<j_{e(v_j)}\leq m$.
 Then, by commutator calculus, for each $v_j\in\calV_{0,\alpha}$ one has the $e(v_j)-1$ relations 
 \[
  \begin{split}
   1&=\left[v_{j_1}v_{j_2}^{-1},v_j   \right] = [w_{j_1}\cdots w_{j_2-1},w_j] \\ &=[w_{j_1},w_j]\cdots [w_{j_2-1},w_j]\cdot y_1\\
   &\;\vdots \\
  1 &=\left[v_{j_{e(v_j)-1}}v_{j_{e(v_j)}}^{-1},v_j\right] = [w_{j_{e(v_j)-1}}\cdots w_{j_{e(v_j)}-1},w_j] 
  \\ &= [w_{j_{e(v_j)-1}},w_j]\cdots [w_{j_{e(v_j)}-1},w_j]  \cdot y_{e(v_j)-1}\\
  \end{split} \]
for some $y_1,\ldots,y_{e(v_j)-1}\in H^{(3)}$.

Now let $F$ be the free pro-$p$ group generated by $\calY$, and for every element $x\in F$, let $\bar x$ denote the image of $x$ via the canonical projection $F\to F/F^{(3)}$.
The list of all the relations above satisfies the hypothesis of Proposition~\ref{prop:cupprod}, as their images modulo $F^{(3)}$
give rise to the subset
\[         
\left\{\begin{array}{c} \overline{[w_i,w_j]},\; \text{with }\{v_i,v_j\}\in\calE_0, \\
 \sum_{k=j_1}^{j_2-1}\overline{[w_k,w_j]},\:\ldots,\:\sum_{k=j_{e(v_j)-1}}^{j_{e(v_j)}-1}\overline{[w_k,w_j]},\;\text{with }
 v_j\in\calV_{0,\alpha}
        \end{array}
\right\}              \subseteq \Phi(F)/F^{(3)}                                              
\]
(here we use the additive notation for $\Phi(F)/F^{(3)}$).
Therefore, the set
\[
 \left\{\begin{array}{c} (\omega_i\vert_H)\smallsmile(\omega_j\vert_H),\;\text{with }\{v_i,v_j\}\in\calE_0, \\
 (\omega_{j_1}\vert_H)\smallsmile(\omega_j\vert_H),\:\ldots,\:(\omega_{j_{e(v_j)-1}}\vert_H)\smallsmile(\omega_j\vert_H),\;\text{with }
 v_j\in\calV_{0,\alpha}
        \end{array}
\right\}
\]
is a linearly independent subset of $\rmH^2(H,\Z/p)$, and thus of $\rmr_{G_\Gamma,H}(\Lambda_1(\Gamma^\ast)\smallsmile V_0)$.
This implies \eqref{eq:dim res 1}.
\end{proof}


\subsection{Proof of Theorem~1.2}\label{ssec:thm1}

First, we prove Theorem~\ref{thm:AbsGalType intro}--(i).

\begin{thm}\label{thm:chordal pabsgaltype}
 Let $\Gamma=(\calV,\calE)$ be a simplicial chordal graph.
 Then the associated pro-$p$ RAAG $G_\Gamma$ is of $p$-absolute Galois type.
\end{thm}

\begin{proof}
Let $\alpha$ be a non-trivial element of $\rmH^1(G_\Gamma,\Z/p)$, and set $N=\Ker(\alpha)$.
Put $\calV=\{v_1,\ldots,v_d\}$ (so that $d=\dd(\Gamma)$), and write
$$\alpha=a_1\chi_1+a_2\chi_2+\ldots+a_d\chi_d,\qquad a_i\in\Z/p.$$
After replacing every generator $v_i$ with $a_i^{-1} v_i$ in $\calV$, if $a_i\neq0,1$, we may assume without loss of generality that $a_i\in\{0,1\}$ for every $i=1,\ldots,d$.
Moreover, after renumbering the vertices of $\Gamma$, we may assume that $\alpha=\chi_1+\ldots+\chi_m$, for some $m\in\{1,\ldots,d\}$.
Finally, let $\Gamma_\alpha=(\calV(\Gamma_\alpha),\calE(\Gamma_\alpha))$ be the induced subgraph of $\Gamma$ with vertices $\calV(\Gamma_\alpha)=\{v_1,\ldots,v_m\}$ (so that $m=\dd(\Gamma_\alpha)$).

By Theorem~\ref{prop:LLSWW}--(i), it sufficies to show that $\Img(c_\alpha)=\Ker(\rmr_{G_{\Gamma},N})$.
In fact, since $\Img(c_\alpha)\subseteq\Ker(\rmr_{G_{\Gamma},N})$ --- as \eqref{eq:exactsequence GGamma} is a complex ---, it is enough to show the left-side inequality in  
\begin{equation}\label{eq:inequality thm 1}
\dim(\Img(c_\alpha))\geq\dim(\Ker(\rmr_{G_\Gamma,N}))= \dim\left(\Lambda_2(\Gamma^\ast)\right)-\dim\left(\Img(\rmr_{G_\Gamma,N})\right).
\end{equation}
Moreover, if $H\subseteq N$ is defined as in \S~\ref{ssec:res}, the functoriality of the restriction map implies that $\dim(\Img(\rmr_{G_\Gamma,N}))\geq \dim(\Img(\rmr_{G_\Gamma,H}))$.
Therefore, showing the inequality
\begin{equation}\label{eq:proof chordal AbsGalType}
\dim\left(\Img(c_\alpha)\right)+\dim\left(\Img(\rmr_{G_\Gamma,H})\right)\geq \dim\left(\Lambda_2(\Gamma^\ast)\right)=|\calE|.
\end{equation}
will prove the left-side inequality in \eqref{eq:inequality thm 1}, and thus the equality $\Img(c_\alpha)=\Ker(\rmr_{G_{\Gamma},N})$.

Let $H_\alpha\subseteq H$ be as defined in \S~\ref{ssec:res}, and let $W_\alpha$ be a subspace of $\rmr_{G_\Gamma,H}(\Lambda_2(\Gamma_\alpha^\ast))$ such that the morphism
$$\rmr_{H,H_\alpha}\vert_{W_\alpha}\colon W_\alpha\longrightarrow\rmr_{G_\Gamma,H_\alpha}(\Lambda_2(\Gamma_\alpha^\ast))$$
is an isomorphism, and let $V_0$ be as in \S~\ref{ssec:cup product}--\ref{ssec:res}.
Then $W_\alpha\cap\rmr_{G_\Gamma,H}(\Lambda_1(\Gamma)\smallsmile V_0)=0$ as $\alpha'\vert_{H_\alpha}=0$ for every $\alpha'\in V_0$, so that $\rmr_{H,H_\alpha}(\alpha''\smallsmile\alpha')=0$ for every $\alpha''\in\Lambda_1(\Gamma^\ast)$.
Hence, 
$$\Img\left(\rmr_{G_\Gamma,H}\right)\supseteq W_\alpha\oplus\rmr_{G_\Gamma,H}(\Lambda_1(\Gamma)\smallsmile V_0),$$
and consequently 
\begin{equation}\label{eq:thm GGamma 3}
\dim\left(\Img(\rmr_{G_\Gamma,H})\right)\geq 
\dim(W_\alpha) + \dim\left(\rmr_{G_\Gamma,H}(\Lambda_1(\Gamma)\smallsmile V_0)\right). 
\end{equation}
Now, by Proposition~\ref{prop:res2} and Proposition~\ref{prop:res 1}, one has
\begin{equation}\label{eq:big sum}
  \begin{split}
 \dim(W_\alpha)+\dim(\rmr_{G_\Gamma,H}(\Lambda_1(\Gamma)\smallsmile V_0))   
  \geq & \left(|\calE(\Gamma_\alpha)|-\dd(\Gamma_\alpha)+\rmr(\Gamma_\alpha)\right) +\\
 &+ \left(|\calE_0|+\sum_{v_j\in\calV_{0,\alpha}}(e(v_j)-1)\right).
 \end{split}
 \end{equation}
Equations \eqref{eq:dim Imcalpha}, \eqref{eq:thm GGamma 3}, and \eqref{eq:big sum}, imply that
\[\begin{split}
   \dim(\Img(c_\alpha))+\dim(\Img(\rmr_{G_\Gamma,H})) &\geq |\calE(\Gamma_\alpha)|+|\calE_0|+|\calV_{0,\alpha}|+\sum_{v_j\in\calV_{0,\alpha}}(e(v_j)-1) \\
   &= |\calE(\Gamma_\alpha)|+|\calE_0|+\sum_{v_j\in\calV_{0,\alpha}}e(v_j)=|\calE|,
  \end{split}\]
and this yields inequality \eqref{eq:proof chordal AbsGalType}.
\end{proof}

\begin{rem}\label{rem:chordalgraphs}\rm
The iterated procedure to construct chordal simplicial graphs (cf. Proposition~\ref{prop:chordal}) makes them --- and the associated pro-$p$ RAAGs, also in the generalized version (see \cite[\S~5.1]{qsv:quadratic} for the definition of {\sl generalized} pro-$p$ RAAG) --- rather special: indeed, by \cite[Prop.~5.22]{qsv:quadratic} a generalized pro-$p$ RAAG associated to a chordal simplicial graph may be constructed by iterating proper amalgamated free pro-$p$ products over uniformly powerful (in some cases, free abelian) subgroups.

\noindent On the one hand, this property is crucial in the proof of Proposition~\ref{prop:res2}; on the other hand this implies that the $\Z/p$-cohomology algebra of a generalized pro-$p$ RAAG associated to a chordal simplicial graph is quadratic (cf. \cite[Rem.~5.25]{qsv:quadratic}) --- notice that, unlike pro-$p$ RAAGs, a generalized pro-$p$ RAAG may yield a non-quadratic $\Z/p$-cohomology algebra (see \cite[Ex.~5.14]{qsv:quadratic}). 

\noindent Finally, the pro-$p$ RAAGs associated to chordal simplicial graphs are precisely those pro-$p$ RAAGs which are {\sl coherent} (cf. \cite[Thm.~1.6]{sz:raags}).
\end{rem}

As stated in Theorem~\ref{thm:AbsGalType intro}, chordal graphs are not the only simplicial graphs yielding pro-$p$ RAAGs of $p$-absolute Galois type.
Let $n$ be a positive integer, and let $\rmQ_n$ be the simplicial graph consisting of a row of $n$ subsequent squares, i.e., $\rmQ_n$ is the graph with geometric realization
 
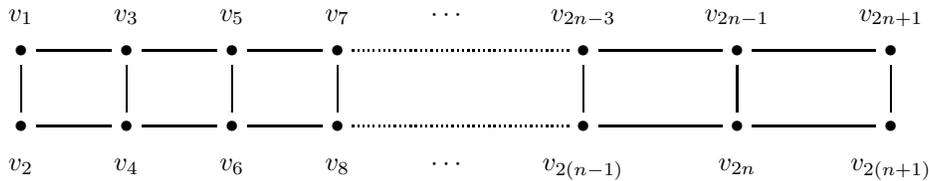
\[
  \xymatrix@R=1.5pt{ v_1 & v_3 & v_5 & v_7 &\cdots & v_{2n-3} & v_{2n-1} & v_{2n+1}  
  \\ \bullet\ar@{-}[r] \ar@{-}[ddd]& \bullet\ar@{-}[r] \ar@{-}[ddd]& \bullet\ar@{-}[r] \ar@{-}[ddd] & 
    \bullet\ar@{.}[rr] \ar@{-}[ddd] && \bullet\ar@{-}[r] \ar@{-}[ddd] &  \bullet\ar@{-}[r] \ar@{-}[ddd] & \bullet\ar@{-}[ddd]  \\  \\  
  \\ \bullet\ar@{-}[r] &\bullet\ar@{-}[r] &\bullet\ar@{-}[r] & \bullet\ar@{.}[rr]  &&\bullet\ar@{-}[r] &\bullet\ar@{-}[r]& \bullet \\
  v_2 & v_4 & v_6 & v_8 &\cdots & v_{2(n-1)} & v_{2n} & v_{2(n+1)} 
  } \]
(for the example displayed in the picture above, $n\geq 3$).
  
Clearly, $\rmQ_n$ is not chordal.
Yet, the structure of such a graph shows a feature similar to the structure of chordal graphs.
Given an induced subgraph $\Gamma'$ of a simplicial graph $\Gamma$, we say that $\Gamma'$ is a {\sl subsquare} of $\Gamma$ if $\Gamma'$ is isomorphic to $\rmQ_1$, and we say that $\Gamma'$ is a {\sl subedge} of $\Gamma$ if $\Gamma'$ consists of two joined vertices, i.e., $\calV(\Gamma')=\{v,w\}$ and $\calE(\Gamma')=\{\{v,w\}\}$.

\begin{lem}\label{lem:subgraphs Q}
 Let $n$ be a positive integer.
 Every connected induced subgraph of $\rmQ_n$ may be constructed recursively by pasting along complete subgraphs (i.e., subgraphs consisting of a single vertex or subedges), starting from single vertices, subedges and subsquares.
\end{lem}

\begin{proof}
 Let $\Gamma'$ be an induced subgraph of $\rmQ_n$, and consider the set $\mathcal{S}$ of those subgraphs of $\Gamma'$ which are either subsquares of $\Gamma'$, or subedges of $\Gamma'$ which are not subedges of any subsquare of $\Gamma'$.
 Then $\Gamma'$ may be constructed by pasting together all subgraphs in $\mathcal{S}$. 
 It is straightforward to see that if $\Gamma_1,\Gamma_2,\Gamma_3\in\mathcal{S}$, and $\Gamma_1\cap\Gamma_i\neq\varnothing$ for both $i=2,3$, then $\Gamma_2\cap\Gamma_3=\varnothing$.
 Moreover, one has the following:
 \begin{itemize}
  \item[(a)] if two subsquares belonging to $\mathcal{S}$ have non-trivial intersection, then they past along a common subedge;\item[(b)] if two subedges belonging to $\mathcal{S}$ have non-trivial intersection, then they past along a common vertex;
  \item[(c)] if a subsquare and a subedge of $\Gamma'$, both belonging to $\mathcal{S}$, have non-trivial intersection, then they past along a common vertex.
 \end{itemize}
This completes the proof.  
\end{proof}

\begin{lem}\label{lem:squares subsquares}
Let $n$ be a positive integer, and let $\Gamma'$ be an induced subgraph of $\rmQ_n$.
Moreover, let $\rmq(\Gamma')$ be the number of distinct subsquares of $\Gamma'$.
Then 
\begin{equation}\label{eq:number subsquares}
 |\calE(\Gamma')|-\dd(\Gamma')+\rmr(\Gamma')=\rmq(\Gamma').
\end{equation}
\end{lem}

\begin{proof}
 We proceed following the inductive construction of $\Gamma'$, as done in Lemma~\ref{lem:subgraphs Q}.
 If $\Gamma'$ consists of a single vertex, of a couple of joined vertices, or if it is a subsquare, then it is straightforward to see that \eqref{eq:number subsquares} holds.
 
 So, suppose that $\Gamma'$ is the pasting of two proper induced subgraphs $\Gamma_1,\Gamma_2$ along a common subgraph $\Delta$,
 where either $\Delta$ is a single vertex, or a subedge of $\Gamma'$.
 Clearly, $\rmq(\Gamma')=\rmq(\Gamma_1)+\rmq(\Gamma_2)$.
 If $\Delta$ is a single vertex, then
 \[
  |\calE(\Gamma')|=|\calE(\Gamma_1)|+|\calE(\Gamma_2)|\qquad\text{and}
  \qquad\dd(\Gamma')=\dd(\Gamma_1)+\dd(\Gamma_2)-1;
 \]
while if $\Delta$ is a subedge then 
  \[
  |\calE(\Gamma')|=|\calE(\Gamma_1)|+|\calE(\Gamma_2)|-1\qquad\text{and}
  \qquad\dd(\Gamma')=\dd(\Gamma_1)+\dd(\Gamma_2)-2.
 \]
 Therefore, if \eqref{eq:number subsquares} holds for both $\Gamma_1,\Gamma_2$, then it holds also for $\Gamma'$.
 Finally, if $\Gamma_1,\ldots,\Gamma_r$ are the connected components of $\Gamma'$, then
 \[
  |\calE(\Gamma')|=|\calE(\Gamma_1)|+\ldots+|\calE(\Gamma_r)|\qquad\text{and}
  \qquad\dd(\Gamma')=\dd(\Gamma_1)+\ldots+\dd(\Gamma_r),
 \]
 so, if \eqref{eq:number subsquares} holds for all the connected components $\Gamma_1,\ldots,\Gamma_r$, then it holds also for $\Gamma'$.
\end{proof}

We are ready to prove Theorem~\ref{thm:AbsGalType intro}--(ii).
  
\begin{thm}
Let $n$ be a positive integer.
The pro-$p$ RAAG $G_{\rmQ_n}$ is of $p$-absolute Galois type.
\end{thm}

\begin{proof}
Put $\rmQ=\rmQ_n$.
We use the same notation as in \S~\ref{ssec:RAAGs pagt}--\ref{ssec:res}, with $\Gamma=\rmQ$.
Let $\alpha$ be a non-trivial element of $\rmH^1(G_{\rmQ},\Z/p)$ --- as done in the proof of Theorem~\ref{thm:chordal pabsgaltype}, we may assume without loss of generality that 
$$ \alpha=\chi_{a_1}+\ldots\chi_{a_m}\qquad\text{for some }1\leq a_1<\ldots<a_m\leq 2(n+1).$$
By Theorem~\ref{prop:LLSWW}--(i) it sufficies to show that $\Img(c_\alpha)=\Ker(\rmr_{G_{{\rmQ}},N})$.

Let $H,H_\alpha\subseteq G_{\rmQ}$ as defined in \S~\ref{ssec:res}.
We pursue the same strategy as in the proof of Theorem~\ref{thm:chordal pabsgaltype}, in order to prove inequality \eqref{eq:proof chordal AbsGalType} for $\Gamma=\rmQ$.
Let $W_\alpha$ be a subspace of $\rmr_{G_{\rmQ},H}(\Lambda_2(\rmQ_\alpha^\ast))$ as defined in the proof of Theorem~\ref{thm:chordal pabsgaltype}.
Then $$\Img(\rmr_{G_{\rmQ},H})\supseteq W_\alpha\oplus \rmr_{G_{\rmQ},H}(\Lambda_1(\rmQ^\ast)\smallsmile V_0).$$
By \eqref{eq:dim Imcalpha} and by Proposition~\ref{prop:res 1} (which hold for any simplicial graph), it is enough to show that
\begin{equation}\label{eq:inequality Q}
 \dim(\rmr_{G_{\rmQ},H_\alpha}(\Lambda_2(\rmQ_\alpha^\ast)))\geq |\calE(\rmQ_\alpha)|-\dd(\rmQ_\alpha)+\rmr(\rmQ_\alpha).
\end{equation}
Let $\rmQ(i)$ be the $i$-th subsquare of $\rmQ$, i.e., $\rmQ(i)$ is the induced subgraph of $\rmQ$ with vertices 
$\calV(\rmQ(i))=\{v_{2i-1},\ldots,v_{2i+2}\}$.
If $\rmQ(i)$ is a subsquare of $\rmQ_\alpha$, too, then by commutator calculus one has the relation
\[
  1=\left[v_{2i-1}v_{2i+2}^{-1},\:v_{2i}v_{2i+1}^{-1}\right]
  = \left[w_{2i-1} w_{2i} w_{2i+1},\:w_{2i}\right]
  =[w_{2i-1},w_{2i}]\cdot[w_{2i+1},w_{2i}]\cdot y_i ,
\]
for some $y_i\in H_{\alpha}^{(3)}$, as both $v_{2i-1},v_{2i+2}$ commute with both $v_{2i},v_{2i+1}$.
Now let $F$ be the free pro-$p$ group generated by $\{w_{a_1},\ldots,w_{a_{m-1}}\}$, and for $x\in F$ let $\bar x$ denote the image of $x$ via the canonical projection $F\to F/F^{(3)}$.
If $\rmQ(i_1),\ldots\rmQ(i_{\rmq(\rmQ_\alpha)})$ are the subsquares of $\rmQ_\alpha$,
the list of $\rmq(\rmQ_\alpha)$ relations above satisfies the hypothesis of Proposition~\ref{prop:cupprod}, as their images modulo $F^{(3)}$ give rise to the set
\[
 \left\{\:\begin{array}{c} 
 \overline{[w_{2i_1-1},w_{2i_1}]}+\overline{[w_{2i_1+1},w_{2i_1}]}\\
 \vdots\\
\overline{[w_{2i_{\rmq(\rmQ_\alpha)}-1},w_{2i_{\rmq(\rmQ_\alpha)}}]}+ 
\overline{[w_{2i_{\rmq(\rmQ_\alpha)}+1},w_{2i_{\rmq(\rmQ_\alpha)}}]}
 \end{array} \:\right\}\subseteq \Phi(F)/F^{(3)}.
\]
(here we use the additive notation for $\Phi(F)/F^{(3)}$).
Therefore, the set 
\[
 \left\{\:\begin{array}{c} 
 (\omega_{2i_1-1}\vert_{H_\alpha})\smallsmile(\omega_{2i_1}\vert_{H_\alpha})=
 \rmr_{G_{\rmQ},H_\alpha}\left(\omega_{2i_1-1}\smallsmile\omega_{2i_1}\right)\\
 \vdots\\
 (\omega_{2i_{\rmq(\rmQ_\alpha)}-1}\vert_{H_\alpha})\smallsmile(\omega_{2i_{\rmq(\rmQ_\alpha)}}\vert_{H_\alpha})=
 \rmr_{G_{\rmQ},H_\alpha}\left(\omega_{2i_{\rmq(\rmQ_\alpha)}-1}\smallsmile\omega_{2i_{\rmq(\rmQ_\alpha)}}\right)
 \end{array} \:\right\}
\]
is a linearly independent subset of $\rmH^2(H_\alpha,\Z/p)$ --- and, in fact, of 
$\rmr_{G_{\rmQ},H_\alpha}(\Lambda_2(\rmQ_\alpha^\ast))$ ---, of cardinality $\rmq(\rmQ_\alpha)$.
Therefore, $\dim(\rmr_{G_{\rmQ},H_\alpha}(\Lambda_2(\rmQ_\alpha^\ast)))\geq \rmq(\rmQ_\alpha)$, and inequality \eqref{eq:inequality Q} follows by Lemma~\ref{lem:squares subsquares}.
\end{proof}

We were not able to prove any result about graphs containing $n$-cycles, with $n\geq5$, as induced subgraphs.
Still, we suspect that the answer to the following question is positive.

\begin{question}
Let $\Gamma=(\calV,\calE)$ be a simplicial graph.
Is it true that the associated pro-$p$ RAAG of $p$-absolute Galois type if, and only if, $\Gamma$ does not contain $n$-cycles  as induced subgraphs for $n\geq5${\rm?}
\end{question}


\section{More pro-$p$ groups of $p$-absolute Galois type}\label{sec:prod}

\subsection{Free pro-$p$ products}\label{ssec:freeprod}

One knows that for every $n\geq 3$ the $n$-Massey vanishing property is preserved by free pro-$p$ products (cf. \cite[Prop.~4.5]{mt:massey}).
We show that also the property of being of $p$-absolute Galois type is preserved by free pro-$p$ products.

\begin{thm}\label{thm:freeprod abs Gal type}
 Let $G_1,G_2$ be pro-$p$ groups, and let $G$ denote their free pro-$p$ product $G_1\amalg G_2$.
 If $G_1$ and $G_2$ are of $p$-absolute Galois type, then also $G$ is of $p$-absolute Galois type.
\end{thm}

\begin{proof}

 Since $G=G_1\amalg G_2$, one has an isomorphism of graded $\Z/p$-algebras 
 \begin{equation}\label{eq:cohomology freeprod}
  \bfH^\bullet(G)\simeq \bfH^\bullet(G_1)\oplus \bfH^\bullet(G_2),
 \end{equation}
which implies that $\alpha_1\smallsmile\alpha_2=0$ for every $\alpha_1\in\rmH^1(G_1,\Z/p)$ and $\alpha_2\in\rmH^1(G_2,\Z/p)$ (cf. \cite[Thm.~4.1.4--4.1.5]{nsw:cohn}).

By Theorem~\ref{prop:LLSWW}--(i), it sufficies to show that \eqref{eq:exactsequence intro} is exact at $\rmH^2(G,\Z/p)$ for every non-trivial $\alpha\in \rmH^1(G,\Z/p)$.
By \eqref{eq:cohomology freeprod} we may write $\alpha=\alpha_1+\alpha_2$, with $\alpha_1=\alpha\vert_{G_1}\in\rmH^1(G_1,\Z/p)$ and $\alpha_2=\alpha\vert_{G_2}\in\rmH^1(G_2,\Z/p)$.
Set $N=\Ker(\alpha)$, and $N_1=\Ker(\alpha_1)=N\cap G_1$ and $N_2=\Ker(\alpha_2)=N\cap G_2$.
By the Kurosh Subgroup Theorem for free pro-$p$ products of pro-$p$ groups (cf. \cite[Thm.~4.2.1]{nsw:cohn}), one has 
\begin{equation}\label{eq:kurosh}
 N=N_1\amalg N_2\amalg \underbrace{\left(\coprod_{x\in\mathcal{S}_1}xN_1x^{-1}\right)\amalg \left(\coprod_{y\in\mathcal{S}_2}yN_2y^{-1}\right)\amalg F}_{H},
\end{equation}
where $\mathcal{S}_1,\mathcal{S}_2$ are finite subsets of non-trivial elements of $G$ (in particular, one has
$xN_1x^{-1}\neq N_1$ and $yN_2y^{-1}\neq N_2$ for every $x\in\mathcal{S}_1$ and $y\in\mathcal{S}_2$), and $F$ is a (possibly trivial) free pro-$p$ group.
Thus, by \eqref{eq:cohomology freeprod} one has 
$$\rmH^2(N,\Z/p)=\rmH^2(N_1,\Z/p)\oplus \rmH^2(N_2,\Z/p)\oplus \rmH^2(H,\Z/p).$$

For $i=1,2$, let $c_i\colon\rmH^1(G_i,\Z/p)\to\rmH^2(G_i,\Z/p)$ denote the map induced by cup-product with $\alpha_i$.
By \eqref{eq:cohomology freeprod}, for every $\alpha'\in\rmH^1(G,\Z/p)$ one has
\[\begin{split}
 c_\alpha(\alpha') &= \left(\alpha'\vert_{G_1}+\alpha'\vert_{G_2}\right)\smallsmile(\alpha_1+\alpha_2) \\
 &= (\alpha'\vert_{G_1}\smallsmile\alpha_1)+(\alpha'\vert_{G_2}\smallsmile\alpha_2) \\
 &=c_1(\alpha'\vert_{G_1})+c_2(\alpha'\vert_{G_2}).
\end{split}\]
Therefore, $\Img(c_\alpha)=\Img(c_1)\oplus\Img(c_2)$.
On the other hand, by \eqref{eq:cohomology freeprod} for every $\beta\in\rmH^2(G,\Z/p)$ one has $\beta=\rmr_{G,G_1}(\beta)+\rmr_{G,G_2}(\beta)$, and since $N_i=N\cap G_i$ for both $i=1,2$, one has 
\[\begin{split}
 \rmr_{G,N}(\beta) &= \rmr_{G,N}\left(\rmr_{G,G_1}(\beta)+\rmr_{G,G_2}(\beta)\right)\\
 &=\rmr_{G_1,N_1}\left(\rmr_{G,G_1}(\beta)\right)+\rmr_{G_2,N_2}\left(\rmr_{G,G_2}(\beta)\right).
\end{split}\]
Consequently, $\Ker(\rmr_{G,N})=\Ker(\rmr_{G_1,N_1})\oplus\Ker(G_2,N_2)$, which is equal to $\Img(c_1)\oplus\Img(c_2)$, as by hypothesis both $G_1,G_2$ are of $p$-absolute Galois type.
This concludes the proof.
\end{proof}

\begin{rem}\rm
 By Theorem~\ref{prop:LLSWW}--(ii), if $G$ is as in Theorem~\ref{thm:freeprod abs Gal type}, then $G$ has the $p$-cyclic Massey vanishing property.
 In fact, employing the universal property of free pro-$p$ products it is easy to prove that also the strong $n$-fold Massey vanishing property, for every $n\geq3$, is preserved by free pro-$p$ products (cf. \cite[Prop.~4.8]{JT:U4}). 
\end{rem}


\subsection{Demushkin groups}\label{ssec:Demushkin}

Recall that a Demushkin group is a pro-$p$ group $G$ satisfying the following:
\begin{itemize}
 \item[(i)] $\dim(\rmH^1(G,\Z/p))<\infty$;
 \item[(ii)] $\rmH^2(G,\Z/p)\simeq\Z/p$;
 \item[(iii)] the cup-product induces a non-degenerate bilinear form\[
            \xymatrix{\rmH^1(G,\Z/p)\times \rmH^1(G,\Z/p)\ar[r]^-{\textvisiblespace\smallsmile\textvisiblespace} & 
            \rmH^2(G,\Z/p)};                  \]
\end{itemize}
cf., e.g., \cite[Def.~3.9.9]{nsw:cohn}.
From condition (ii), one deduces that Demushkin groups have a minimal presentation with only one defining relation (cf. \S~\ref{ssec:cohomology}).
In particular, the only finite Demushkin group $G$ occurs in case $p=2$ and $\dim(\rmH^1(G,\Z/p))=1$, i.e., $G\simeq\Z/2$ (cf. \cite[Prop.~3.9.10]{nsw:cohn}).
For more properties of Demushkin groups we direct the reader to \cite[Ch.~III, \S~9]{nsw:cohn}.

One knows that any Demushkin group has the strong $n$-Massey vanishing property for every $n\geq3$ (cf. \cite[Thm.~3.5]{pal:massey} and \cite[Prop.~4.1]{JT:U4}).
We show that, in addition, any Demushkin group is of $p$-absolute Galois type.

\begin{thm}\label{thm:Demushkin}
 Let $G$ be a Demushkin group.
 Then $G$ is of $p$-absolute Galois type.
\end{thm}

\begin{proof}
Condition~(iii) in the definition of Demushkin group implies that for any $\alpha\in\rmH^1(G,\Z/p)$, $\alpha\neq0$, one has
\begin{equation}\label{eq:cup H1H2 Demushkin}
 \alpha\smallsmile\rmH^1(G,\Z/p)=\rmH^2(G,\Z/p).
\end{equation}
Put $N=\Ker(\alpha)$.
Then by \eqref{eq:cup H1H2 Demushkin}, the map $c_\alpha\colon\rmH^1(G,\Z/p)\to\rmH^2(G,\Z/p)$ is surjective.
Thus, by \eqref{eq:complex} one has $\Img(c_\alpha)=\Ker(\rmr_{G,N})$, so that \eqref{eq:exactsequence intro} is exact at $\rmH^2(G,\Z/p)$.
By Theorem~\ref{prop:LLSWW}--(i), this is sufficient to show that $G$ is of $p$-absolute Galois type.
Still, here we provide an explicit proof of the fact that, if $\dim(\rmH^1(G,\Z/p))$ is even, the sequence \eqref{eq:exactsequence intro} is exact at $\rmH^1(G,\Z/p)$ for every $\alpha\in\rmH^1(G,\Z/p)$.
(Recall that if $p\neq2$ then $\dim(\rmH^1(G,\Z/p))$ is necessarily even.)

So, let $G$ be a Demushkin group with $d=\dim(\rmH^1(G,\Z/p))$ even, and pick $\alpha\in\rmH^1(G,\Z/p)$, $\alpha\neq0$.
Since the cup-product induces a non-degenerate bilinear form, $\rmH^1(G,\Z/p)$ decomposes as a direct sum of hyperbolic planes, and thus we may complete $\{\alpha\}$ to a basis $\{\alpha_1=\alpha,\alpha_2,\ldots,\alpha_d\}$ of $\rmH^1(G,\Z/p)$ such that 
\begin{equation}\label{eq:cupprod Demushkin cor}
\alpha_1\smallsmile\alpha_2=\alpha_3\smallsmile\alpha_4=\ldots=\alpha_{d-1}\smallsmile\alpha_d\neq0,\qquad\text{and}\qquad \alpha_i\smallsmile\alpha_j=0\end{equation}
for any other couple of $i,j$ with $1\leq i<j\leq d$.
Thus, $\Ker(c_{\alpha_1})=\Span\{\alpha_1,\alpha_3,\ldots,\alpha_d\}$.

Let $\{y_1,\ldots,y_d\}$ be a minimal generating set of $G$ such that $\alpha_i(y_j)=\delta_{ij}$.
Then by \eqref{eq:cupprod Demushkin cor} one has an equivalence
\begin{equation}\label{eq:rel Demushkin mod F3}
 [y_1,y_2][y_3,y_4]\cdots[y_{d-1},y_d]\cdot \prod_{i=1}^dy_i^{pb_i}\equiv 1\mod F^{(3)},
\end{equation}
for some $b_1,\ldots,b_d\in\Z/p$ (cf. \cite[Prop.~3.9.13--(ii)]{nsw:cohn}).

Set $N=\Ker(\alpha_1)$ --- so, $N$ is generated as a normal subgroup of $G$ by the set $\{y_1^p,y_2,\ldots,y_d\}$. 
Then $N$ is again a Demushkin group, with $$\dim\left(\rmH^1(N,\Z/p)\right)=2+p(d-1)$$ (cf. \cite[Thm.~3.9.15]{nsw:cohn}).
Moreover, \eqref{eq:rel Demushkin mod F3} implies that $[y_2,y_1]\equiv y_1^{pb_1}\bmod\Phi(N)$, and therefore the set 
\[
\calY=\left\{\:y_1^p,\:y_2,\:y_1^\nu y_iy_1^{-\nu}\:\mid\:3\leq i\leq d,\:0\leq \nu\leq p-1\:\right\}
\]
is a minimal generating set of $N$.
Let $\{\psi_1,\psi_2,\psi_{3,0},\psi_{3,1},\ldots,\psi_{d,p-1}\}$ be the basis of $\rmH^1(N,\Z/p)$ dual to $\calY$, and consider $\{1,y_1,\ldots,y_1^{p-1}\}$ as a set of representatives of the quotient $G/N$.
For every $\alpha'\in\rmH^1(N,\Z/p)$ and every $x\in G$ one has the formula
\begin{equation}\label{eq:cor formula}
 \mathrm{cor}_{N,G}^1(\alpha')(x)=\sum_{h=0}^{p-1}\alpha'\left(y^{-h'}xy_1^h\right),
\end{equation}
where $0\leq h'\leq p-1$ is such that $xy_1^hN=y_1^{h'}N$ (cf. \cite[Ch.~I, \S~5.4]{nsw:cohn}).
Then \eqref{eq:cor formula} implies that $\mathrm{cor}_{N,G}^1(\psi_1)=\alpha_1$ and $\mathrm{cor}_{N,G}^1(\psi_{i,\nu})=\alpha_i$ for every $3\leq i\leq d$ and $0\leq \nu\leq p-1$, while $\mathrm{cor}_{N,G}^1(\psi_2)=0$.
Namely, $\Img(\mathrm{cor}_{N,G}^1)=\Ker(c_{\alpha_1})$.
\end{proof}

\begin{rem}\label{rem:Demushkin Galois}\rm
If $\K$ is a $p$-adic local field containing a root of 1 of order $p$, then its maximal pro-$p$ Galois group is a Demushkin group, with $$\dim(\rmH^1(G_{\K}(p),\Z/p))=[\K:\Q_p]+2$$
(cf. \cite[Thm.~7.5.11]{nsw:cohn}).
Also, $\Z/2$ is the maximal pro-$2$ Galois group of $\dbR$.
It is still an open problem to determine whether {\sl any other} Demushkin group occurs as the maximal pro-$p$ Galois group of a field containing a root of 1 of order $p$.
\end{rem}


\subsection{Direct products}\label{ssec:directprod}

{Let $G$ be a pro-$p$ group whose abelianization $G/G'$ is a free abelian pro-$p$ group.
If $p=2$, then the natural map $\rmH^1(G,\Z/4)\to\rmH^1(G,\Z/2)$ is surjective, so that $\alpha\smallsmile\alpha=0$ for every $\alpha\in\rmH^1(G,\Z/2)$ (cf., e.g., \cite[Fact.~7.1]{qw:cyc}; if $p>2$ then one has $\alpha\smallsmile\alpha=0$ trivially).

\begin{prop}\label{prop:cyc torfree}
Let $G$ be a pro-$p$ group whose abelianization $G/G'$ is a free abelian pro-$p$ group.
Then for every $\alpha\in\rmH^1(G,\Z/p)$ and for every $n\geq2$, the $n$-fold Massey product $\langle\alpha,\ldots,\alpha\rangle$ vanishes.
\end{prop}

\begin{proof}
 By Proposition~\ref{prop:massey cup}--(i), we may suppose that $\alpha\neq0$.
 Let $\pi\colon G\to G/G'$ denote the canonical projection, and let $\bar\alpha\colon G/G'\to\Z/p$ be the morphism such that $\alpha=\bar\alpha\circ\pi$.
 Moreover, pick $g\in G$ such that $\alpha(g)=1$. 
 Then one has $$G/G'=\langle\:\pi(g)\:\rangle\times B,\qquad\text{for some }B\subseteq\Ker(\bar\alpha),$$
 while $\langle\pi(g)\rangle\simeq\Z_p$.
 Let $\rho'\colon G/G'\to\dbU_{n+1}$ be the representation such that
 \[
 \rho'(\pi(g))=\left(\begin{array}{ccccc} 1 & 1 &0 &\cdots &0 \\ & 1 & 1 &\ddots&\vdots \\ && \ddots & \ddots&0 \\ &&& 1&1 \\&&&&1 \end{array}
\right)
\]
and $\rho'\vert_B\equiv I_{n+1}$.
Then, the composition $\rho=\rho'\circ\pi\colon G\to\dbU_{n+1}$ is a homomorphism satisfying $\rho_{i,i+1}=\alpha$ for every $i=1,\ldots,n$, and Proposition~\ref{prop:masse unip} yields the claim.
\end{proof}
}

Let $G_1$ and $G_2$ be two pro-$p$ groups, and let $G=G_1\times G_2$ be their direct product.
Then for the $\Z/p$-cohomology algebra of $G$ one has the following:
\begin{equation}\label{eq:directprod cohom}
 \begin{split}
  \rmH^1(G,\Z/p) &= \rmH^1(G_1,\Z/p)\oplus \rmH^1(G_2,\Z/p),\\
  \rmH^2(G,\Z/p) &= \rmH^2(G_1,\Z/p)\oplus \rmH^2(G_2,\Z/p)\oplus\left(\rmH^1(G_1,\Z/p)\wedge \rmH^1(G_2,\Z/p)\right),
 \end{split}
\end{equation}
(cf. \cite[Ch.~II, \S~4, Thm.~2.4.6 and Ex.~7]{nsw:cohn}).
In particular, if $\{\chi_1,\ldots,\chi_{d_1}\}$ and $\{\psi_1,\ldots,\psi_{d_2}\}$ are bases of $\rmH^1(G_1,\Z/p)$ and $\rmH^1(G_2,\Z/p)$ respectively, then 
$$\left\{\:\chi_i\smallsmile\psi_j\:\mid\:1\leq i\leq d_1,\:1\leq j\leq d_2\:\right\}$$
is a basis of $\rmH^1(G_1,\Z/p)\wedge \rmH^1(G_2,\Z/p)$.

\begin{thm}\label{thm:directprod massey}
 Let $G_1,G_2$ be two pro-$p$ groups with torsion-free abelianization, and set $G=G_1\times G_2$.
\begin{itemize}
 \item[(i)] If both $G_1,G_2$ have the $n$-Massey vanishing property for every $n\geq 3$, then also $G$ has the $n$-Massey vanishing property for every $n\geq 3$.
 \item[(ii)] If both $G_1,G_2$ have the $p$-cyclic Massey vanishing property, then also $G$ has the $p$-cyclic Massey vanishing property.
\end{itemize}
\end{thm}

\begin{proof}
First of all, observe that by \eqref{eq:directprod cohom} if $\alpha,\beta\in\rmH^1(G,\Z/p)$ have trivial cup-product $\alpha\smallsmile\beta$, and $\alpha,\beta$ do not lie in the same subgroup $\rmH^1(G_i,\Z/p)$, then necessarily $\beta=a\alpha$ for some $a\in\Z/p$.
In this case, by Proposition~\ref{prop:massey cup}--(i) we may assume that $a=1$ or $a=0$.
Moreover, obviously the abelianization $G/G'\simeq G_1/G_1'\times G_2/G_2'$ of $G$ is a free abelian pro-$p$ group.
\medskip

\noindent(i)\quad Let $\alpha_1,\ldots,\alpha_n$ be a sequence of elements of $\rmH^1(G,\Z/p)$ such that the $n$-fold Massey product $\langle\alpha_1,\ldots,\alpha_n\rangle$ is defined --- by Proposition~\ref{prop:massey cup}, we may assume that $\alpha_i\neq0$ for all $i=1,\ldots,n$.
Then $\alpha_i\smallsmile\alpha_{i+1}=0$ for all $i=1,\ldots,n-1$, and by Proposition~\ref{prop:masse unip}--(i) there exists a representation $\bar\rho\colon G\to\bar\dbU_{n+1}$ such that $\rho_{i,i+1}=\alpha_i$ for all $i=1,\ldots,n$.

If $\alpha_1,\ldots,\alpha_n\in\rmH^1(G_1,\Z/p)$, then the $n$-fold Massey product $\langle\alpha_1\vert_{G_1},\ldots,\alpha_n\vert_{G_2}\rangle$ is defined in $\bfH^\bullet(G_1)$ --- indeed, $\bar\rho\vert_{G_1}\colon G_1\to\bar\dbU_{n+1}$ is a representation satisfying $(\bar\rho\vert_{G_1})_{i,i+1}=\alpha_i\vert_{G_1}$ --- and thus, by hypothesis, $\langle\alpha_1\vert_{G_1},\ldots,\alpha_n\vert_{G_2}\rangle$ vanishes, yielding a representation $\rho'\colon G_1\to\dbU_{n+1}$ satisfying $\rho'_{i,i+1}=\alpha_i\vert_{G_1}$ for every $i=1,\ldots,n$.
Then, the representation $\rho\colon G\to \dbU_{n+1}$ given by $\rho\vert_{G_1}=\rho'$ and $G_2\subseteq \Ker(\rho)$ satisfies $\rho_{i,i+1}=\alpha_i$ for every $i=1,\ldots,n$, and thus $\langle\alpha_1,\ldots,\alpha_n\rangle$ vanishes by Proposition~\ref{prop:masse unip}.
Analogously, if $\alpha_1,\ldots,\alpha_n\in\rmH^1(G_2,\Z/p)$ then $\langle\alpha_1,\ldots,\alpha_n\rangle$ vanishes.

Otherwise, we may assume that $\alpha_1=\ldots=\alpha_n$, and the claim follows by Proposition~\ref{prop:cyc torfree}.
\medskip

\noindent(ii)\quad Pick two non-trivial elements $\alpha,\beta\in\rmH^1(G,\Z/p)$ such that $\alpha\smallsmile\beta=0$.

If $\alpha,\beta\in\rmH^1(G_1,\Z/p)$, then by hypothesis the $p$-fold Massey product $$\langle\alpha\vert_{G_1},\ldots,\alpha\vert_{G_1},\beta\vert_{G_1}\rangle$$ vanishes in $\bfH^\bullet(G_1)$, as 
$$(\alpha\vert_{G_1})\smallsmile(\beta\vert_{G_1})=\rmr_{G,G_1}(\alpha\smallsmile\beta)=0$$ (by the functoriality of the restriction map).
Thus, by Proposition~\ref{prop:masse unip}--(ii) there exists a representation $\rho'\colon G_1\to\dbU_{p+1}$ satisfying $\rho'_{i,i+1}=\alpha$ for $i=1,\ldots,p+1$, and $\rho'_{p,p+1}=\beta$.
As done above, we may define a representation $\rho\colon G\to\dbU_{p+1}$ such that $\rho\vert_{G_1}=\rho'$ and $G_2\subseteq\Ker(\rho)$, so that $\rho_{i,i+1}=\alpha$ for $i=1,\ldots,p-1$ and $\rho_{p,p+1}=\beta$.
Therefore, the $p$-fold Massey product $\langle\alpha,\ldots,\alpha,\beta\rangle$ vanishes in $\bfH^\bullet(G)$.
Analogously, if $\alpha,\beta\in\rmH^1(G_2,\Z/p)$ then the $p$-fold Massey product $\langle\alpha,\ldots,\alpha,\beta\rangle$ vanishes in $\bfH^\bullet(G)$.

Otherwise, if $\beta=a\alpha$ for some $a\in(\Z/p)^\times$, then we may assume that $a=1$, and the $p$-fold Massey product $\langle\alpha,\ldots,\alpha\rangle$ vanishes in $\bfH^\bullet(G)$ by Proposition~\ref{prop:cyc torfree}. 
\end{proof}

The {\sl Elementary Type Conjecture} on maximal pro-$p$ Galois groups, formulated by I.~Efrat, predicts that the maximal pro-$p$ Galois group of a field containing a root of 1 of order $p$ may be constructible --- if it is finitely generated --- starting from free pro-$p$ groups and Demushkin groups (and also the cyclic group of order 2, if $p=2$), and iterating free pro-$p$ products and certain semidirect products with $\Z_p$ (cf. \cite{ido:etc,ido:etc2}, see also \cite[\S~7.5]{qw:cyc}). 
In case of fields containing {\sl all} roots of 1 of $p$-power order, then a finitely generated maximal pro-$p$ Galois group shoud be constructible starting from free pro-$p$ groups and Demushkin groups {\sl with torsion-free abelianization}, iterating free pro-$p$ products and {\sl direct} products with $\Z_p$.
Therefore, from Theorem~\ref{thm:directprod massey}--(ii), together with the aforementioned results contained in \cite[\S~4]{mt:massey}, one deduces the following.

\begin{cor}\label{cor:ETC}
If the Elementary Type Conjecture holds true, then the maximal pro-$p$ Galois group $G_{\K}(p)$ of a field $\K$ containing all roots of 1 of $p$-power order has the strong $n$-Massey vanishing property for every $n\geq3$, provided that $G_{\K}(p)$ is finitely generated.
\end{cor}

In other words, Efrat's Elementary Type Conjecture implies a strengthened version of Mina\v{c}-T\^an's conjecture \cite[Conj.~1.1]{mt:conj} --- formulated again by Mina\v{c} and T\^an in \cite{JT:U4} (cf. Remark~\ref{rem:cupdef}, see also \cite{pal:massey}) --- for fields containing all roots of 1 of $p$-power order with finitely generated maximal pro-$p$ Galois group.

\begin{thm}\label{thm:directprod pabsgaltype}
 Let $G_1,G_2$ be two finitely generated pro-$p$ groups with torsion-free abelianization and such that
 $\rmH^2(G_i,\Z/p)=\rmH^1(G_i,\Z/p)\smallsmile\rmH^1(G_i,\Z/p)$ for both $i=1,2$.
If $G_1$ and $G_2$ are of $p$-absolute Galois type, then also the direct product $G_1\times G_2$ is of $p$-absolute Galois type.
\end{thm}

\begin{proof}
Let $\alpha$ be a non-trivial element of $\rmH^1(G,\Z/p)$, and set $N=\Ker(\alpha)$.

Suppose first that $\alpha\in\rmH^1(G_1,\Z/p)$.
Then $N=N_1\times G_2$, where $N_1=N\cap G_1=\Ker(\alpha\vert_{G_1})$, so that 
\[
 \rmH^2(N,\Z/p)=\rmH^2(N_1,\Z/p)\oplus\rmH^2(G_2,\Z/p)\oplus\left(\rmH^1(N_1,\Z/p)\wedge\rmH^1(G_2,\Z/p)\right)
\]
by \eqref{eq:directprod cohom}.
Let $V_0$ be a subspace of $\rmH^1(G_1,\Z/p)$ such that $\rmH^1(G_1,\Z/p)=V_0\oplus\Span\{\alpha\}$.
We decompose the map $\rmr_{G,N}$ into its restrictions to the direct summands of $\rmH^2(G,\Z/p)$ as follows:
\[
\xymatrix@C=0pt{  \rmH^2(G_1,\Z/p)\ar[d]^-{\rmr_{G_1,N_1}} \ar@<-4ex>@/_2pc/[d]|-{\rmr_{G,N}}   &\oplus&
\rmH^2(G_2,\Z/p)\ar@{=}[d]^-{\mathrm{Id}}
&\oplus& \left(V_0\wedge\rmH^1(G_2,\Z/p)\right)\oplus 
\left(\alpha\wedge\rmH^1(G_2,\Z/p)\right)\ar@<-8ex>[d]^-{\rmr'}\ar@<8ex>[d]^-{0}
\\  
\rmH^2(N_1,\Z/p)&\oplus&\rmH^2(G_2,\Z/p)&\oplus&\left(\rmH^1(N_1,\Z/p)\wedge\rmH^1(G_2,\Z/p)\right).
}\]
By hypothesis, $\Ker(\rmr_{G_1,N_1})=\Img(c_{\alpha\vert_{G_1}})$; while the map 
$$\rmr'=(\res_{G_1,N_1}^1\vert_{V_0})\wedge\mathrm{Id}\colon V_0\wedge \rmH^1(G_2,\Z/p)\longrightarrow\rmH^1(N_1,\Z/p)\wedge\rmH^1(G_2,\Z/p)$$ is injective, as $\Ker(\res_{G_1,N_1}^1)=\Span\{\alpha\}$.
Altogether, $\Ker(\rmr_{G,N})=\Img(c_\alpha)$.

Obviously, after switching $G_1$ and $G_2$ the same argument shows that $\Ker(\rmr_{G,N})=\Img(c_\alpha)$ if $\alpha\in\rmH^1(G_2,\Z/p)$.

Suppose now that $\alpha=\alpha_1+\alpha_2$, with $\alpha_1=\alpha\vert_{G_1}\in\rmH^1(G_1,\Z/p)$ and $\alpha_2=\alpha\vert_{G_2}\in\rmH^1(G_2,\Z/p)$.
Let $\calX=\{x_1,\ldots,x_{d_1}\}$ and $\calY=\{y_1,\ldots,y_{d_2}\}$ be minimal generating sets of $G_1$ and $G_2$ respectively, satisfying $\alpha_1(x_{d_1})=1$ and $\alpha(x_i)=0$ for $i\neq d_1$, and $\alpha_2(y_{d_2})=1$ and $\alpha(y_j)=0$ for $j\neq d_2$; and let 
\[\begin{split}
\mathcal{B}_1&=\left\{\:\chi_1,\:\ldots,\:\chi_{d_1}=\alpha_1\:\right\}\subseteq \rmH^1(G_1,\Z/p),\\ 
\mathcal{B}_2&=\left\{\:\psi_1,\:\ldots,\:\psi_{d_2}=\alpha_2\:\right\}\subseteq\rmH^1(G_2,\Z/p),
  \end{split}
  \]
be the bases dual to $\calX$ and $\calY$ respectively.
First, observe that $\alpha_1\smallsmile\alpha_2=\alpha\smallsmile\alpha_2=\alpha_1\smallsmile\alpha$, so that $\alpha_1\smallsmile\alpha_2\in\Img(c_\alpha)$.
Moreover, for every $1\leq i<d_1$ and $1\leq j<d_2$ one has { 
\begin{equation}\label{eq:directprod 1}
 \begin{split}
 \chi_i\smallsmile\alpha&=\chi_i\smallsmile\alpha_1+\chi_i\smallsmile\alpha_2=
 \chi_i\smallsmile\chi_{d_1}+\chi_i\smallsmile\psi_{d_2},\\
    \alpha\smallsmile\psi_j&=\alpha_1\smallsmile\psi_j+\alpha_2\smallsmile\psi_j=
    \chi_{d_1}\smallsmile\psi_j-\psi_j\smallsmile\psi_{d_2}.
  \end{split}
\end{equation}
We claim that the set
$$\mathcal{S}=\left\{\:\alpha_1\smallsmile\alpha_2,\:\chi_i\smallsmile\alpha,\:\alpha\smallsmile\psi_j\:\mid\: 1\leq i<d_1,\:1\leq j<d_2\:\right\}$$
is a linearly independent subset of $\rmH^2(G,\Z/p)$.
Indeed, suppose that there exist $a,b_i,c_j\in\Z/p$, with $1\leq i<d_1$ and $1\leq j<d_2$, such that
\begin{equation}\label{eq:directprod 2}
 a(\alpha_1\smallsmile\alpha_2)+\sum_{i=1}^{d_1-1}b_i(\chi_i\smallsmile\alpha)+\sum_{j=1}^{d_2-1}c_j(\alpha\smallsmile\psi_j)=0
\end{equation}
Applying \eqref{eq:directprod 1} to \eqref{eq:directprod 2} yields
\[\begin{split}
&\sum_{i=1}^{d_1-1}b_i(\chi_i\smallsmile\chi_{d_1})-\sum_{j=1}^{d_2-1}c_j(\psi_j\smallsmile\psi_{d_2})\:+ \\
&+\left(\sum_{i=1}^{d_1-1}b_i(\chi_i\smallsmile\psi_{d_2})+\sum_{j=1}^{d_2-1}c_j(\chi_{d_1}\smallsmile\psi_{j})+
a(\chi_{d_1}\smallsmile\psi_{d_2}) \right) =0.
\end{split}\]
Since the three sets $\{\chi_i\smallsmile\chi_j\:\mid\:1\leq i<j\leq d_2\}$, $\{\psi_i\smallsmile\psi_j\:\mid\:1\leq i<j\leq d_2\}$ and $\{\chi_i\smallsmile\psi_j\:\mid\:1\leq i\leq d_1,1\leq j\leq d_2\}$ are bases of the direct summands of $\rmH^2(G,\Z/p)$ by \eqref{eq:directprod cohom}, one has $a=b_i=c_j=0$ for every $i,j$.
Therefore, \begin{equation}\label{eq:dim imcalpha directprod}
            \dim(\Img(c_\alpha))\geq|\mathcal{S}|= (d_1-1)+(d_2-1)+1=d_1+d_2-1.
           \end{equation}}

Now put $z=x_{d_1}y_{d_2}^{-1}$ and 
$$\mathcal{Z}_1=\left\{\:x_1,\:\ldots,\:x_{d_1-1},\:z\:\right\},\qquad 
\mathcal{Z}_2=\left\{\: y_1,\:\ldots,\:y_{d_2-1},\:z\:\right\}.$$
Let $H_1$, $H_2$, and $H$, be the subgroups of $G$ generated by $\mathcal{Z}_1$, by $\mathcal{Z}_2$, and by $\mathcal{Z}_1\cup\mathcal{Z}_2$ respectively. 
Then $H_1,H_2\subseteq H\subseteq N$.
Since $\mathcal{Z}_1\cup\mathcal{Z}_2\cup\{x_{d_1}\}$ is a minimal generating set of $G$, the inclusions $H_i\hookrightarrow G$ (with $i\in\{1,2\}$) and $H\hookrightarrow G$ induce morphisms of $p$-elementary abelian groups $H_i/\Phi(H_i)\to G/\Phi(G)$ and $H/\Phi(H)\to G/\Phi(G)$ which are injective.
Therefore, $\mathcal{Z}_i$ is a minimal generating set of $H_i$, and $\mathcal{Z}_1\cup\mathcal{Z}_2$ is a minimal geberating set of $H$, so that by duality the sets 
\[\begin{split}
   \mathcal{B}_1&=\left\{\:\chi_1\vert_{H_1},\:\ldots,\:\chi_{d_1-1}\vert_{H_1},\:\chi_{d_1}\vert_{H_1}=\alpha_1\vert_{H_1}\:\right\},\\
\mathcal{B}_2&=\left\{\:\psi_1\vert_{H_2},\:\ldots,\:\psi_{d_2-1}\vert_{H_2},-\psi_{d_2}\vert_{H_2}=-\alpha_2\vert_{H_2}\:\right\},\\
 \mathcal{B}_H&=\left\{\:\chi_1\vert_{H},\:\ldots,\:\chi_{d_1-1}\vert_{H},\:\psi_1\vert_H,\:\ldots,\:\psi_{d_2-1}\vert_H,\:\alpha_1\vert_{H}=-\alpha_2\vert_{H}\:\right\}
  \end{split} 
\]
are bases of $\rmH^1(H_1,\Z/p)$, $\rmH^1(H_2,\Z/p)$, and $\rmH^1(G,\Z/p)$ respectively.

We claim that $H_1\simeq G_1$ and $H_2\simeq G_2$. 
Indeed, put $C=\langle y_{d_2}\rangle\simeq\Z_p$, and $K_1=H_1C$.
Since $G/G'=G_1/G_1'\times G_2/G_2'\simeq\Z_p^{d_1+d_2}$ by hypothesis, one has 
$$H_1/H_1'\simeq \Z_p^{d_1}\qquad\text{and}\qquad K_1/K_1'= H_1/H_1'\times C\simeq\Z_p^{d_1+1},$$
and consequently $H_1\cap C=\{1\}$ --- namely, $K_1=H_1\times C$.
On the other hand, $K_1$ is the subgroup of $G$ generated by $\{x_1,\ldots,x_{d_1},y_{d_2}\}$, i.e., $K_1=G_1\times C$.
Altogether, one has
\begin{equation}
 \xymatrix{ H_1=\dfrac{H_1}{H_1\cap C}\ar[r]^-{\tau_1} & \dfrac{H_1\times C}{C}=\dfrac{G_1\times C}{C} &
 G_1\ar[l]_-{\tau_2} }
\end{equation}
where both $\tau_1$ and $\tau_2$ are isomorphisms.
Put $\phi=\tau_2^{-1}\circ\tau_1$.
Then $\phi$ is an isomorphism, and in particular $\phi(x_i)=x_i$ for $i=1,\ldots,d_1-1$, and $\phi(z)=\tau_2^{-1}(zC)=\tau_2^{-1}(x_{d_1}C)=x_{d_1}$.
By duality, the isomorphism $\phi^\ast\colon \rmH^1(G_1,\Z/p)\to\rmH^1(H_1,\Z/p)$ induced by $\phi$ coincides with the restriction of $\res_{G,H_1}^1$ to $\rmH^1(G_1,\Z/p)$ --- in particular, $\phi^\ast(\chi_{d_1})=\alpha_1\vert_{H_1}=\chi_{d_1}\vert_{H_1}$.
Since $\rmH^2(G_1,\Z/p)=\rmH^1(G_1,\Z/p)\smallsmile\rmH^1(G_1,\Z/p)$, by \eqref{eq:cup res} also the restriction of $\rmr_{G,H_1}$ to $\rmH^2(G_1,\Z/p)$ is an isomorphism, so that if $\mathcal{S}_1=\{\chi_i\smallsmile\chi_j\mid(i,j)\in\mathcal{I}_1\}$ is a basis of $\rmH^2(G_1,\Z/p)$, then 
$$\rmr_{G,H_1}(\mathcal{S}_1)=\left\{\:(\chi_i\vert_{H_1})\smallsmile(\chi_j\vert_{H_1})\:\mid\:(i,j)\in\mathcal{I}_1\:\right\}$$
is a basis of $\rmH^2(H_1,\Z/p)$.
Thus, by the functoriality of the restriction map $\rmr_{G,H}(\mathcal{S}_1)$ is a linearly independent subset of $\rmH^2(H,\Z/p)$.
An analogous argument proves that if $\mathcal{S}_2=\{\psi_i\smallsmile\psi_j\mid(i,j)\in\mathcal{I}_2\}$ is a basis of $\rmH^2(G_2,\Z/p)$, then $\rmr_{G,H_2}(\mathcal{S}_2)$ is a basis of $\rmH^2(H_1,\Z/p)$, and $\rmr_{G,H}(\mathcal{S}_2)$ is a linearly independent subset of $\rmH^2(H,\Z/p)$.

Finally, in $H$ one has $(d_1-1)(d_2-1)$ relations $[x_i,y_j]=1$, with $i<d_1$ and $j<d_2$, which give rise to the subset
\[
\mathcal{S}_H= \left\{\:(\chi_i\vert_H)\smallsmile(\psi_j\vert_H)\:\mid\:1\leq i< d_1,\:1\leq j<d_2\:\right\}
\]
of $\rmH^2(H,\Z/p)$, which is linearly independent by Proposition~\ref{prop:cupprod}.

Altogether, the disjoint union 
$$\rmr_{G,H}(\mathcal{S}_1)\:\dot\cup\:\rmr_{G,H}(\mathcal{S}_2)\:\dot\cup\:\mathcal{S}_H\subset\rmH^2(H,\Z/p)$$ 
is a linearly independent subset of $\rmH^2(H,\Z/p)$, as each one of the three subsets is linearly independent, and one has $\rmr_{G,H_j}(\mathcal{S}_i)=\rmr_{H,H_i}(\mathcal{S}_H)=\{0\}$, for $i,j\in\{1,2\}$, $i\neq j$ --- while $\rmr_{G,H_i}(\mathcal{S}_i)$ is linearly independent for both $i=1,2$.
Therefore, by the functoriality of the restriction map one has 
\begin{equation}\label{eq:dim directprod res}
\begin{split}
\dim\left(\Img(\rmr_{G,N})\right)&\geq\dim\left(\Img(\rmr_{G,H})\right)\\ &\geq \dim\left(\rmH^2(G_1,\Z/p)\right)+\dim\left(\rmH^2(G_2,\Z/p)\right)+(d_1-1)(d_2-1).                                 \end{split}
\end{equation}
Summing up \eqref{eq:dim imcalpha directprod} and \eqref{eq:dim directprod res}, together with \eqref{eq:directprod cohom}, yields 
\[\begin{split}
   \dim\left(\Img(c_\alpha)\right)+ \dim\left(\Img(\rmr_{G,N})\right)&\geq \dim\left(\rmH^2(G_1,\Z/p)\right)+\dim\left(\rmH^2(G_2,\Z/p)\right)+d_1d_2 \\ &=\dim\left(\rmH^2(G,\Z/p)\right),
  \end{split}
\]
and therefore $\Ker(\rmr_{G,N})=\Img(c_\alpha)$. 
Then Theorem~\ref{prop:LLSWW}--(i) yields the claim.
\end{proof}

The list of finitely generated pro-$p$ groups $G$ of $p$-absolute Galois type satisfying the three conditions (a)--(c) in Theorem~\ref{thm:products intro}--(ii) includes:
\begin{itemize}
 \item[(a)] pro-$p$ RAAGs associated to simplicial graphs --- whose family includes, in turn, finitely generated free pro-$p$ groups and finitely generated free abelian pro-$p$ groups ---, by Theorems~\ref{thm:massey intro}--\ref{thm:AbsGalType intro};
 \item[(b)] Demushkin groups with torsion-free abelianization, or, equivalently, pro-$p$ completions of oriented surface groups --- namely, pro-$p$ groups with minimal presentation
 \[
  G=\left\langle\:x_1,\ldots,x_d\:\mid\:[x_1,x_2][x_3,x_4]\cdots[x_{d-1},x_d]\:\right\rangle
 \]
with $d$ even (these are precisely the Demushkin groups $G$ with invariant $q(G)=0$, cf. \cite{labute}) ---, by \cite[Thm.~4.3]{mt:massey} and Theorem~\ref{thm:Demushkin}.
\end{itemize}
There are only very few ways to combine these pro-$p$ groups via direct product, to obtain the maximal pro-$p$ Galois group of a field containing a root of 1 of order $p$, as stated by the following (see also \cite[Prop.~3.2]{koenig} and \cite[Cor.~1.2]{ilirslobo}).

\begin{prop}\label{prop:directprod Galois}
 Let $G_1,G_2$ be two pro-$p$ groups.
 Then the direct product $G_1\times G_2$ occurs as the maximal pro-$p$ Galois group of a field containing a root of 1 of order $p$ only if one of the two factors is a free abelian pro-$p$ group, and 
 the other factor occurs as the maximal pro-$p$ Galois group of a field containing all roots of 1 of $p$-power order.
\end{prop}

\begin{proof}
 Suppose that $G$ occurs as the maximal pro-$p$ Galois group of a field $\K$ containing a root of 1 of order $p$.
 Then by the Fundamental Theorem of Galois theory, for both $i=1,2$, $G_i$ is isomorphic to $G_{\K_i}(p)$, with $\K_i=\K(p)^{G_i}$ --- and clearly both fields $\K_1,\K_2$ contain a root of 1 of order $p$.
 
 By \cite[Thm.~5.6]{cq:bk}, the direct product of two pro-$p$ groups may occur as the maximal pro-$p$ Galois group of a field containing a root of 1 of order $p$ only if one of the two factors --- say, $G_2$ in our case --- is a free abelian pro-$p$ group. Now suppose that there is a root $\zeta$ of 1 of order $p^f$, with $f\geq2$, lying in $\K(p)$ but not in $\K$.
Then $\K(\zeta)/\K$ is a Galois extension, of degree at most $p^{f-1}$.
Let $g$ be an element of $G$ such that $g.\zeta=\zeta^{1+p^{f-1}}$ --- i.e., $g$ surjects to a suitable generator of the subquotient 
 \[
  \Gal(\K(\zeta)/\K(\zeta^p))=\frac{G_{\K(\zeta^p)}(p)}{G_{\K(\zeta)}(p)}\simeq \Z/p
 \]
of $G$.
Since $G_2$ is an abelian normal subgroup of $G$, \cite[Thm.~7.7]{qw:cyc} implies that 
$$ghg^{-1}=h^{1+p^{f-1}}\qquad\text{for every }h\in G_2.$$
This contradicts the fact that, by hypothesis, $G_2$ is contained in the center of $G$, as $f\geq2$.
Thus, $\K$ (and hence also $\K_1$) contains every root of 1 of $p$-power order lying in $\K(p)$.

Conversely, if $G_1\simeq G_{\bar\K}(p)$ for some field $\bar\K$ containing all roots of 1 of $p$-power order, and $G_2$ is a free abelian pro-$p$ group, then it is well-known that $G$ occurs as the maximal pro-$p$ Galois group of the field of Laurent series $\bar\K(\!(\mathcal{X})\!)$, where $\calX=\{X_i\:\mid\:i\in \mathcal{I}\}$ is a set of indeterminates in bijection with a minimal generating set of $G_2$ (cf., e.g., \cite[Ex.~4.10]{cq:bk}).
\end{proof}

Altogether, Theorems~\ref{thm:massey intro}--\ref{thm:products intro} and Theorem~\ref{thm:directprod massey}
provide a huge amount of pro-$p$ groups of $p$-absolute Galois type with the $n$-Massey vanishing property for every $n\geq3$.
Still, only few of them occur as the maximal pro-$p$ group of a field containing a root of 1 of order $p$, because of the restrictions given by Theorem~\ref{thm:ilirpavel} and Proposition~\ref{prop:directprod Galois}.
This yields Corollary~\ref{cor:noGal}.

\subsection{Pro-$p$ groups hereditarily of $p$-absolute Galois type}\label{ssec:hered}

 We say that a pro-$p$ group $G$ is {\sl hereditarily of $p$-absolute Galois type} if every closed subgroup of $G$ is of $p$-absolute Galois type.
Clearly, the maximal pro-$p$ Galois group of a field containing a root of 1 of order $p$ is hereditarily of $p$-absolute Galois type, as every closed subgroup is again the maximal pro-$p$ Galois group of a field containing a root of 1 of order $p$.

The following result shows that, in order to verify the hereditariety, it is enough to check open subgroups, in analogy with the Bloch-Kato property and 1-cyclotomicity (cf. \cite[Cor.~3.2 and Cor.~3.5]{qw:cyc}).

\begin{prop}\label{prop:hereditarily}
 Let $G$ be a pro-$p$ group.
 If every open subgroup of $G$ is of $p$-absolute Galois type, then $G$ is hereditarily of $p$-absolute Galois type.
\end{prop}

\begin{proof}
 Let $H$ be a closed subgroup of $G$.
 By Theorem~\ref{prop:LLSWW}-(i), it is enough to show that for every non-trivial $\alpha\in\rmH^1(H,\Z/p)$, one has $\Ker(\rmr_{H,N_\alpha})=\Img(c_\alpha)$, where $N_\alpha=\Ker(\alpha)$ and $c_\alpha\colon \rmH^1(H,\Z/p)\to\rmH^2(H,\Z/p)$ denotes, as usual, the map induced by the cup-product by $\alpha$.
 
 Since $H$ is a closed subgroup of $G$, one has $H=\bigcap_{U\in\mathcal{U}_H}U$, where $\mathcal{U}_H$ is the set of all open subgroups of $G$ containing $H$ (cf. \cite[Prop.~1.2--(iii)]{ddsms}), and thus 
 \begin{equation}\label{eq:cohom H and U serre}
  \rmH^n(H,\Z/p)=\varinjlim_{U\in\mathcal{U}_H}\rmH^n(U,\Z/p)\qquad\text{for every }n\geq 1,
 \end{equation}
where the morphisms of the injective limit are given by the restriction maps $\res_{U,V}^n$ for every $U,V\in\mathcal{U}_H$ such that $U\supseteq V\supseteq H$ (cf. \cite[Ch.~I, \S~2.2, Prop.~8]{serre:galc}).
In particular, for $U\in\mathcal{U}_H$ sufficiently small, there exists $\alpha_U\in\rmH^1(U,\Z/p)$ such that the restriction $\alpha_U\vert_H$ is $\alpha$.
Then one has a commutative diagram
\begin{equation}\label{eq:commdiag H U}
 \xymatrix@C=1.3truecm{ \rmH^1(U,\Z/p)\ar[r]^-{c_{\alpha_U}}\ar[d]^{\mathrm{res}_{U,H}^1} &
 \rmH^2(U,\Z/p)\ar[r]^-{\rmr_{U,N_{\alpha_U}}}\ar[d]^-{\rmr_{U,H}} &
 \rmH^2(N_{\alpha_U},\Z/p)\ar[d]^-{\rmr_{H_{\alpha_U},H_\alpha}} \\
 \rmH^1(H,\Z/p)\ar[r]^{c_{\alpha}} & \rmH^2(H,\Z/p)\ar[r]^-{\rmr_{H,N_\alpha}} & \rmH^2(N_{\alpha},\Z/p)
 }
\end{equation}
where $c_{\alpha_U}$ denotes the map induced by the cup-product by $\alpha_U$, and $N_{\alpha_U}=\Ker(\alpha_U)$, and the top row is exact by hypothesis.

Now pick $\beta\in\Ker(\rmr_{H,N})$, $\beta\neq0$.
After taking $U$ even smaller, we may assume that there exists $\beta_U\in\rmH^2(U,\Z/p)$ such that $\beta=\rmr_{U,H}(\beta_U)$ and $\rmr_{U,N_{\alpha_U}}(\beta_U)=0$.
Since the top row of \eqref{eq:commdiag H U} is exact, there exists $\alpha'\in\rmH^1(U,\Z/p)$ such that $\beta_U=\alpha'\smallsmile\alpha_U$, and thus $\beta=(\alpha'\vert_H)\smallsmile\alpha$ by \eqref{eq:cup res}.
Hence $\beta\in\Img(c_\alpha)$.
\end{proof}

{\small \subsection*{Acknowledgements}
The authors are deeply indebted with Preston~Wake, as this paper originates from some private discussions between him and the third-named author --- in particular, he suggested the definition of pro-$p$ groups hereditarily of $p$-absolute Galois type.
Moreover, the authors thank Jan~Mina\v{c} and Nguyen~Duy~T\^an, for their comments, and the first- and second-named authors thank the Department of Mathematics of the University of Zaragoza, for the hospitality enjoyed during their graduate studies.

\noindent
The core of this paper was conceived --- and discussed within the authors --- during the conference ``New trends around profinite groups'', hosted by {\sl Centro Internazionale per la Ricerca Matematica}--CIRM, which took place in Levico Terme (Italy) in Sept. 2021: so, the authors wish to thank CIRM together with Eloisa~Detomi, Benjamin~Klopsch, Peter~Symonds, Thomas~S.~Weigel, and Christopher~Voll, for their effort in the organization in such hard times.

\noindent
Last, but not least, the authors thank the anonymous referee, for their careful work with the manuscript, and their numerous useful comments and suggestions, which led to the improvement of the paper.
}

\end{document}